\documentclass[a4paper,oneside, reqno]{amsart}
\usepackage{mathrsfs}
\usepackage{amsfonts}
\usepackage{amssymb}
\usepackage{amsxtra}
\usepackage{dsfont}
\usepackage{hyperref}

\usepackage[english,polish]{babel}

\newcommand{\pl}[1]{\foreignlanguage{polish}{#1}}

\usepackage{color}
\usepackage{enumitem}

\theoremstyle{plain}
\newtheorem{theorem}{Theorem}
\newtheorem{conjecture}{Conjecture}
\newtheorem{proposition}{Proposition}[section]

\newtheorem{lemma}[proposition]{Lemma}
\theoremstyle{definition}

\numberwithin{equation}{section}

\newcounter{thm}

\theoremstyle{plain}

\newcommand{\RR}{\mathbb{R}}
\newcommand{\ZZ}{\mathbb{Z}}
\newcommand{\TT}{\mathbb{T}}
\newcommand{\CC}{\mathbb{C}}
\newcommand{\NN}{\mathbb{N}}

\newcommand{\calC}{\mathcal{C}}

\newcommand{\calF}{\mathcal{F}}

\newcommand{\calM}{\mathcal{M}}

\newcommand{\mm}{\mathfrak{m}}

\newcommand{\ind}[1]{{\mathds{1}_{{#1}}}}
\newcommand{\dist}{\operatorname{dist}}

\newcommand{\norm}[1]{{\left| #1 \right|}}
\newcommand{\sprod}[2] {{#1 \cdot #2}}

\newcommand{\abs}[1]{{\left\lvert {#1} \right\rvert}}

\newcommand{\sgn}{\mathrm{sgn}}

\newcommand{\dif}{\mathrm{d}}

\allowdisplaybreaks

\title[Some remarks on dimension-free estimates]
{Some remarks on dimension-free estimates for the
discrete Hardy--Littlewood maximal functions}

\author{Dariusz Kosz}
\address{Dariusz Kosz\\
	Wroc{\l}aw University of Science and Technology\\
	Wybrze{\.z}e Wyspia{\'n}skiego 27\\
	50-370 Wroc{\l}aw, Poland}
\email{dariusz.kosz@pwr.edu.pl}

\author{Mariusz Mirek}
\address{Mariusz Mirek \\
  Department of Mathematics\\
  Rutgers University\\
Piscataway, NJ 08854\\ USA \&
	Instytut Matematyczny\\
	Uniwersytet \pl{Wroc{\lll}awski}\\
	Plac Grun\-waldzki 2/4\\
	50-384 \pl{Wroc{\lll}aw}\\
	Poland}
\email{mariusz.mirek@rutgers.edu}

\author{Pawe{\l} Plewa}
\address{Pawe{\l} Plewa\\
	Wroc{\l}aw University of Science and Technology\\
	Wybrze{\.z}e Wyspia{\'n}skiego 27\\
	50-370 Wroc{\l}aw, Poland}
\email{pawel.plewa@pwr.edu.pl}

\author{B{\l}a{\.z}ej Wr{\'o}bel}
\address{ B{\l}a{\.z}ej Wr{\'o}bel\\
	Instytut Matematyczny\\
	Uniwersytet \pl{Wroc{\lll}awski}\\
	Plac Grun\-waldzki 2/4\\
	50-384 \pl{Wroc{\lll}aw}\\
	Poland}
\email{blazej.wrobel@math.uni.wroc.pl}

\thanks{ Dariusz Kosz and Pawe\l{} Plewa were supported by funds of
Faculty of Pure and Applied Mathematics, Wroc\l{}aw University of
Science and Technology, \#049U/0052/19.  Mariusz Mirek was partially
supported by Department of Mathematics at Rutgers University. Mariusz
Mirek and B{\l}a{\.z}ej Wr{\'o}bel were supported by the National
Science Centre, Poland, grant Opus 2018/31/B/ST1/00204 }

\begin{document}
 
\selectlanguage{english}

\begin{abstract}
Dependencies of the optimal constants in strong and weak type bounds
will be studied between maximal functions corresponding to the
Hardy--Littlewood averaging operators over convex symmetric bodies
acting on $\mathbb R^d$ and $\mathbb Z^d$.  Firstly, we show, in the
full range of $p\in[1,\infty]$, that these optimal constants in
$L^p(\mathbb R^d)$ are always not larger than their discrete analogues
in $\ell^p(\mathbb Z^d)$; and we also show that the equality holds for
the cubes in the case of $p=1$. This in particular implies that the
best constant in the weak type $(1,1)$ inequality for the discrete
Hardy--Littlewood maximal function associated with centered cubes in
$\mathbb Z^d$ grows to infinity as $d\to\infty$, and if $d=1$ it is
equal to the largest root of the quadratic equation $12C^2-22C+5=0$.
Secondly, we prove dimension-free estimates for the
$\ell^p(\mathbb Z^d)$ norms, $p\in(1,\infty]$, of the discrete
Hardy--Littlewood maximal operators with the restricted range of
scales $t\geq C_q d$ corresponding to $q$-balls,
$q\in[2,\infty)$. Finally, we extend the latter result on
$\ell^2(\mathbb Z^d)$ for the maximal operators restricted to dyadic
scales $2^n\ge C_q d^{1/q}$.
\end{abstract}

\maketitle

\section{Introduction}
\subsection{A brief overview of the paper} Throughout this paper $d\in\NN$ always denotes the dimension of the Euclidean space $\RR^d$, and  $G$ denotes a convex symmetric body in $\RR^d$,
which is a bounded closed and symmetric convex subset of $\RR^d$ with
nonempty interior. We shall consider convex bodies in two contexts,
continuous and discrete. Therefore, in
order to avoid unnecessary technicalities, we always assume that $G\subset \RR^d$ is closed (whereas in the
literature it is usually assumed to be open). One of the most
classical examples of convex symmetric bodies are the $q$-balls 
$B^q\subset \RR^d$, $q\in[1,\infty]$, defined for $q \in [1,\infty)$ by
\begin{equation*}
B^q:=B^q(d):=\Big\{x\in\RR^d: \vert x\vert_q:=\Big( \sum_{i=1}^d |x_i|^q\Big)^{1/q}\le 1 \Big\},
\end{equation*}
and for $q=\infty$ by
\begin{equation*}
B^\infty:=B^\infty(d):=\Big\{x\in\RR^d: \vert x\vert_\infty:=\max_{1\leq i\leq d}|x_i|\le 1 \Big\}.
\end{equation*}
If $p=2$ then $B^2$ is the closed unit Euclidean ball in $\RR^d$ centered at the
origin, and if $p=\infty$ then $B^{\infty}$ is the cube in $\RR^d$ centered
at the origin and of side length $2$.

We associate with a convex symmetric body $G\subset \RR^d$ the families of
continuous $(M_t^G)_{t>0}$ and discrete  $(\calM_t^G)_{t>0}$ averaging operators given respectively  by
\begin{equation}
\label{eq:18}
M^G_t F(x):=\frac{1}{|G_t|} \int_{G_t} F(x-y)\, {\rm d} y, \qquad F\in L^1_{\rm loc}(\RR^d),
\end{equation}
and
\begin{equation}
\label{eq:19}
\calM^G_t f(x):=\frac{1}{|G_t\cap\ZZ^d|} \sum_{y\in G_t\cap\ZZ^d } f(x-y), \qquad f\in  \ell^\infty(\ZZ^d),
\end{equation}
where $G_t=\{y\in\RR^d: t^{-1}y\in G\}$ is the dilate of $G\subset \RR^d$. Moreover, we define the corresponding maximal functions by
\begin{equation*}
M_\ast^G F(x) :=\sup_{t>0} \big|M_t^G F(x)\big|,
\qquad \text{ and } \qquad
\calM_\ast^G f(x) :=\sup_{t>0} \big|\calM_t^G f(x)\big|.
\end{equation*}

It is well known that both maximal functions are of weak type $(1,1)$
and of strong type $(p,p)$ for any $p\in(1,\infty]$. Moreover, neither
of these maximal functions is of strong type $(1,1).$ Our primary interest is
focused on determining whether the constants arising in the weak and
strong type inequalities can be chosen independently of the dimension
$d$. Moreover, we shall compare the best constants in such
inequalities for $M_\ast^G$ and $\calM_\ast^G$, respectively.

For
$p\in (1, \infty]$ we denote by $C(G,p)$ the smallest constant $0< C< \infty$ for
which the following strong type inequality holds
\begin{equation*}
\| M_\ast^G F \|_{L^p(\RR^d)}  \leq C \| F \|_{L^p(\RR^d)}, \qquad F \in L^p(\RR^d).
\end{equation*}
Similarly, $C(G,1)$ will stand for the smallest constant $0< C< \infty$ satisfying
\begin{equation*}
\sup_{\lambda>0}\lambda \,  |\{ x \in \RR^d : M_\ast^G F(x) > \lambda \}| \leq C \| F \|_{L^1(\RR^d)}, \qquad F \in L^1(\RR^d).
\end{equation*}
Analogously to $C(G,p)$, we define $\calC(G,p)$ for any $p \in [1, \infty]$, referring to $\calM_\ast^G$ in place  of $M_\ast^G$. 

Our main result of this paper can be formulated as follows. 

\begin{theorem}\label{thm:T}
Fix $d \in \NN$ and let $G \subset \RR^d$ be a convex symmetric
body. Then for each $p \in [1, \infty]$ we have
\begin{align}
\label{eq:16}
C(G,p) \leq \calC(G,p).
\end{align}
 Moreover, for the $d$-dimensional cube
 $B^\infty(d)\subset \RR^d$ one has
 \begin{align}
 \label{eq:52}
 C(B^\infty(d),1) = \calC(B^\infty(d),1).
 \end{align}
\end{theorem}

Some comments are in order:
\begin{itemize}
\item[(i)] Clearly, $C(G,\infty)=\calC(G,\infty)=1$, since we have been working with averaging operators.
\item[(ii)] Theorem \ref{thm:T} gives us a quantitative dependence
between $C(G,p)$ and $\calC(G,p)$. Inequality \eqref{eq:16} coincides
with a well known phenomenon in harmonic analysis, which states that
it is harder to establish bounds for discrete operators than the
bounds for their continuous counterparts.
\item[(iii)] Formula \eqref{eq:52} was observed by the first author in
his master thesis. However, it has not been published before. In
particular, it yields that
$\calC(B^\infty(d),1)\ _{\overrightarrow{d\to\infty}} \infty$ in view
of the result of Aldaz \cite{Ald1} asserting that the optimal constant
$C(B^\infty(d),1)$ in the weak type $(1,1)$ inequality grows to
infinity as $d\to\infty$. Quantitative bounds for the constant
$C(B^\infty(d),1)$ were given by Aubrun \cite{Aub1}, who proved
$C(B^\infty(d),1)\gtrsim_{\varepsilon}(\log d)^{1-\varepsilon}$ for
every $\varepsilon>0$, and soon after that by Iakovlev and Str\"omberg
\cite{IakStr1} who considerably improved Aubrun's lower bound by
showing that $C(B^\infty(d),1)\gtrsim d^{1/4}$. The latter result also
ensures in the discrete setup that
$\calC(B^\infty(d),1)\gtrsim d^{1/4}$.
\item[(iv)] If  $d=1$, then \eqref{eq:52} combined with the result of Melas \cite{Mel} implies that
$\calC(B^\infty(1),1)$ is equal to the larger root of the quadratic equation
$12C^2-22C+5=0$.
\item[(v)] The product structure of the cubes $B^\infty(d)$ and the
fact that one works with continuous/discrete norms for $p=1$ are
essential to prove $C(B^\infty(d),1) = \calC(B^\infty(d),1)$.  At this
moment it does not seem that our method can be used to attain the
equality in \eqref{eq:16} for $G=B^\infty(d)$ with
$p\in(1, \infty)$. In general case, as we shall see later in this
paper, inequality \eqref{eq:16} cannot be reversed.
\item[(vi)] The proof of inequality \eqref{eq:16} relies on a suitable
generalization of the ideas described in the master thesis of the first
author. The details are presented in Section \ref{sec:Tr}.

\end{itemize}

Systematic studies of dimension-free estimates in the continuous case were initiated
by Stein \cite{SteinMax}, who showed that $C(B^2,p)$ is bounded
independently of the dimension for all $p\in(1,\infty]$. Shortly
afterwards Bourgain \cite{B1} proved that $C(G,2)$ can be estimated by
an absolute constant independent of the dimension and the convex
symmetric body $G\subset \RR^d$. This result was extended to the range
$p\in(3/2,\infty]$ in \cite{B2} and independently by Carbery in
\cite{Car1}. It is conjectured that one can estimate $C(G,p)$ by a
dimension-free constant for all $p\in(1,\infty]$. This was verified
for the $q$-balls $B^q$, $q\in[1,\infty)$, by M\"uller \cite{Mul1} and
for the cubes $B^\infty$ by Bourgain \cite{B3}. The latter result
exhibits an interesting phenomenon, which shows that the
dimension-free estimates on $L^p(\RR^d)$ for $p\in(1, \infty]$ cannot
be extended to the weak type $(1, 1)$ endpoint.  Namely, the optimal
constant $C(B^\infty(d),1)$ in the weak type $(1,1)$ inequality, as
Aldaz \cite{Ald1} proved, grows to infinity with the dimension.
Additionally, if the range of scales $t$ in the definition of the
maximal operator $M_\ast^G$ is restricted to the dyadic values ($t\in \mathbb D := \{2^n : n \in \ZZ \}$), then
the constants in the strong type $(p,p)$ inequalities with
$p\in(1,\infty]$ are bounded uniformly in $d$ for any $G$, see
\cite{Car1}. For a more detailed account of the subject in the
continuous case, its history and extensive literature we refer to
\cite{DGM1} or \cite{BMSW4}.

Surprisingly, in the discrete setting there is no hope for estimating
$\calC(G,p)$ independently of the dimension and the convex symmetric
body. Fixing $1\leq \lambda_1<\cdots<\lambda_d<\ldots<\sqrt{2}$ and examining   the ellipsoids 
\begin{equation}
\label{eq:elip}
E(d):=\Big\{x\in \RR^d : \sum_{j=1}^d \lambda_j^2\, x_j^2\,\le 1 \Big\},
\end{equation}
it was proved in \cite[Theorem~2]{BMSW3} that for every $p\in(1, \infty)$ there
is a constant $C_p>0$ such that for all $d\in \NN$ we have
\begin{align}
\label{eq:17}
\mathcal C(E(d), p)\ge \sup_{\|f\|_{\ell^p(\ZZ^d)}\le 1}\big\|\sup_{0<t\le d }|\calM_{t}^{E(d)} f|\big\|_{\ell^p(\ZZ^d)}
\ge
C_p(\log d)^{1/p}.
\end{align}
This inequality shows that if $p\in(3/2, \infty)$, then for sufficiently large $d$ the inequality
in \eqref{eq:16} with $G=E(d)$ is strict,
since from \cite{B2} and \cite{Car1} we know that there exists a
finite constant $C_p>0$ independent of the dimension such that
\begin{align}
\label{eq:32}
C(E(d), p)\le C_p.
\end{align}

On the other hand, for cubes $G=B^\infty(d)$, it
was also proved \cite[Theorem~3]{BMSW3} that for every
$p\in(3/2, \infty]$ there is a finite constant $C_p>0$ such that for
every $d\in\NN$ one has
\begin{align}
\label{eq:31}
\calC(B^\infty(d), p)\le C_p.
\end{align}
For $p\in(1, 3/2]$ it still remains open whether
$\calC(B^\infty(d), p)$ can be estimated independently of the
dimension. In view of the second part of Theorem \ref{thm:T}
interpolation does not help, since
$\calC(B^\infty(d),1)\ _{\overrightarrow{d\to\infty}} \infty$.

Inequalities \eqref{eq:17} and \eqref{eq:31} illustrate that the
dimension-free phenomenon in the discrete setting contrasts sharply
with the situation that we know from the continuous setup. However, as
it was shown in \cite[Theorem~2]{BMSW3}, if the dimension-free
estimates fail in the discrete setting it may only happen for small
scales, see also \eqref{eq:17}.  To be more precise, define the discrete restricted maximal
function
\begin{align}
\label{eq:24}
\calM_{\ast, > D}^G f(x) :=\sup_{t>D} \big|\calM_t^G f(x)\big|; \qquad x\in\ZZ^d, \quad D\ge 0,
\end{align}
corresponding to the averages from \eqref{eq:19}. For $D=0$ we see that $\calM_{\ast, > D}^G=\calM_{\ast}^G$.
Then by \cite[Theorem~1]{BMSW3} one has for arbitrary convex symmetric bodies
$G\subset\RR^d$  that there
exists $c(G)>0$ such that
\begin{align}
\label{eq:26}
\big\|\calM_{\ast, > c(G)d}^G f\big\|_{\ell^p(\ZZ^d)}
\leq
e^6 C(G, p)\|f\|_{\ell^p(\ZZ^d)},
\qquad f\in \ell^p(\ZZ^d).
\end{align}
Specifically, $\frac{1}{2}d^{1/2}\le c(E(d))\le d^{1/2}$ for the ellipsoids \eqref{eq:elip},  and consequently 
by \eqref{eq:26}, we also have
\begin{equation*}
\big\Vert\calM_{\ast, > d^{3/2}}^{E(d)} f\big\Vert_{\ell^p(\ZZ^d)}\leq e^6C(E(d), p)\Vert f\Vert_{\ell^p(\ZZ^d)}, \qquad f\in \ell^p(\ZZ^d),
\end{equation*}
which ensures dimension-free estimates for any $p\in (3/2, \infty]$ thanks to \eqref{eq:32}.

In the case of $q$-balls $G=B^q(d)$, $q\in[1,\infty]$, in view of \cite{B3} and \cite{Mul1}, inequality \eqref{eq:26}  comes down to
\begin{equation}
\label{eq:28}
\big\Vert\calM_{\ast, > d^{1+1/q}}^G f\big\Vert_{\ell^p(\ZZ^d)}\leq C_{p,q}\Vert f\Vert_{\ell^p(\ZZ^d)}, \qquad f\in \ell^p(\ZZ^d),
\end{equation}
for all $p\in(1,\infty]$, where $C_{p,q}$ denotes a constant that
depends on $p$ and $q$ but not on the dimension $d$. 
We close this discussion by gathering  a few conjectures that arose upon completing 
\cite{BMSW3}, \cite{BMSW4} and \cite{BMSW2}. 

\begin{conjecture}
\label{con:1}
Let $d\in\NN$ and let $B^q(d)\subset\RR^d$ be a $q$-ball.
\begin{enumerate}[label*={\arabic*}.]
\item (Weak form) Is it true that for
every $p\in(1, \infty)$ and $q\in [1, \infty]$ there exist constants  $C_{p, q}>0$ and $t_q>0$  such that
for
every $d\in\NN$  we have
\begin{align}
\label{eq:25}
\sup_{\|f\|_{\ell^p(\ZZ^d)}\le 1}\big\|\calM_{\ast, > t_q d}^{B^q(d)} f\big\|_{\ell^p(\ZZ^d)}\le C_{p, q} \, ? 
\end{align}
\item (Strong form)  Is it true that for
every $p\in(1, \infty)$ and $q\in[1, \infty]$  there exists a constant $C_{p, q}>0$ such that for
every $d\in\NN$  we have
\begin{align}
\label{eq:33}
\calC(B^q(d), p)\le C_{p, q} \, ?
\end{align}
\end{enumerate}
\end{conjecture}
A few comments about these conjectures are in order.
\begin{itemize}
\item[(i)] The first conjecture arose on the one hand in view of
the inequalities \eqref{eq:17} and \eqref{eq:28}, and on the other hand in
view of the result from \cite{BMSW4}, where it had been verified for
the Euclidean balls. Indeed, if $q=2$ then following \cite[Section
5]{BMSW4} one can see that \eqref{eq:28} holds with $ad$ in place of
$d^{1+1/q}$, where $a>0$ is a large absolute constant independent of
$d$. Since the dimension-free phenomenon may only break down for small
scales, the conjectured threshold from which we can expect
dimension-free estimates in the discrete setup is at the level of a
constant multiple of $d$.

\item[(ii)] The second conjecture says that one should expect
dimension-free estimates for $\calC(B^q(d), p)$ corresponding to
$q$-balls as we have for their continuous counterparts $C(B^q(d), p)$.
It was verified \cite{BMSW3} in the case of cubes
$\calC(B^{\infty}(d), p)$ with $p\in(3/2, \infty]$ as we have seen in
\eqref{eq:31}. Moreover, in \cite{BMSW2} the second and fourth author
in collaboration with Bourgain and Stein proved that the discrete
dyadic Hardy--Littlewood maximal function
$\sup_{n\in\NN}|\calM^{B^2}_{2^n} f|$ over the Euclidean balls
$B^2(d)$ has dimension-free estimates on $\ell^2(\ZZ^d)$. Although
this can be thought of as the first step towards establishing
\eqref{eq:33}, the general case seems to be very difficult even for
the Euclidean balls or cubes for $p\in(1, 3/2]$. This will surely
require new methods.
\end{itemize}

In our second main result of this paper we verify the first conjecture for the balls $B^q$ for all $q\in(2,\infty)$.

\begin{theorem}\label{thm:1}
For every $q\in(2,\infty)$ and each $a>0$ and $p\in(1,\infty)$ there exists $C(p,q,a)>0$ independent of the dimension $d\in\NN$ such that
for all $f\in\ell^p(\ZZ^d)$ we have
\begin{equation*}
\big\|\sup_{N \geq a d}|\calM_N^{B^q}f|\big\|_{\ell^p(\ZZ^d)}\leq C(p,q,a) \|f\|_{\ell^p(\ZZ^d)}.
\end{equation*}
\end{theorem} 

The proof of Theorem \ref{thm:1} is presented in Section \ref{sec:Bq};
it relies on the methods developed in \cite[Section 5]{BMSW4}, Hanner's inequality \eqref{eq:Han_ineq} and
Newton's generalized binomial theorem. It follows from
\cite[Theorem~2]{BMSW4} that Theorem \ref{thm:1} remains true for
$q=2$, but only for $a>0$, which is large enough. Our proof of Theorem \ref{thm:1}
can be easily adapted to yield the same result. We point out necessary changes in the proof. Since we rely on Hanner's inequality our proof of Theorem \ref{thm:1} does not carry over to $q \in [1,2).$ This is because for such values of $q$ the inequality \eqref{eq:Han_ineq} is reversed.

Our final result concerns the dyadic maximal operator associated with $q$-balls.
\begin{theorem}
\label{thm:10}
Fix $q \in [2, \infty)$. Let $C_1, C_2>0$ and define
$\mathbb D_{C_1, C_2}:=\{N\in\mathbb D:C_1d^{1/q}\le N\le C_2d\}$. Then
there exists a constant $C_q>0$ independent of the dimension such that
for every $f\in \ell^2(\ZZ^d)$ we have
\begin{align}
\label{eq:20}
\big\|\sup_{N\in \mathbb D_{C_1, C_2}}|\mathcal M_N^{B^q}f|\big\|_{\ell^2(\ZZ^d)}\le C_q\|f\|_{\ell^2(\ZZ^d)}.
\end{align}
\end{theorem}
Theorem \ref{thm:10} is an incremental step towards establishing the
second conjecture. By adapting the ideas developed in \cite{BMSW2} we
are able to obtain dimension-free estimates for the discrete
restricted dyadic Hardy--Littlewood maximal functions over $q$-balls
for all $q\in[2, \infty)$.  Theorem \ref{thm:10} generalizes
\cite[Theorem 2.2]{BMSW3}, which was stated for $q = 2$.  The proof of
inequality \eqref{eq:20} as in \cite{BMSW2} exploits the invariance of
$B_N^q\cap\ZZ^d$ under the permutation group of $\NN_d$.  Then we can
 efficiently use probabilistic arguments on a permutation group
corresponding to $\NN_d$ that reduce the matter to the decrease dimension trick
as in \cite{BMSW2}. The proof of Theorem \ref{thm:10} is a technical
elaboration of the methods from \cite{BMSW2}. However, for the
convenience of the reader, mainly due to intricate technicalities we decided
to provide necessary details in Section \ref{sec:4}. We remark that the condition $q \in [2, \infty)$ cannot be dropped in our proof of  Theorem \ref{thm:10} as it is required in the estimate at the origin from Proposition \ref{prop:1}.

\subsection{Notation} The following basic notation will be used 
throughout the paper.

\begin{enumerate}[label*={\arabic*}.]
\item We will write $A \lesssim_{\delta} B$
($A \gtrsim_{\delta} B$) to say that there is an absolute constant
$C_{\delta}>0$ (which depends on a parameter $\delta>0$) such that
$A\le C_{\delta}B$ ($A\ge C_{\delta}B$).  We will write
$A \simeq_{\delta} B$ when $A \lesssim_{\delta} B$ and
$A\gtrsim_{\delta} B$ hold simultaneously. We shall abbreviate subscript $\delta$ if irrelevant. 
\item 
Let $\NN:=\{1,2,\ldots\}$ be the set of positive integers and $\NN_0 := \NN\cup\{0\}$, and
$\mathbb D:=\{2^n: n\in\ZZ\}$ will denote the set of dyadic numbers.
We set $\NN_N := \{1, 2, \ldots, N\}$ for any $N \in \NN$. 
\item For a measurable set
$A\subseteq \RR^d$ we denote by $|A|$ the Lebesgue measure of $A$ and by
$|A\cap \ZZ^d|$ the number of lattice points in $A.$

\item 
The Euclidean space $\RR^d$
is endowed with the standard inner product
\[
x\cdot\xi:=\langle x, \xi\rangle:=\sum_{k=1}^dx_k\xi_k
\]
for every two vectors $x=(x_1,\ldots, x_d)$ and
$\xi=(\xi_1, \ldots, \xi_d)$ from $\RR^d$. Let $|x|_2:=\sqrt{\langle x, x\rangle}$ denote  the Euclidean norm of a vector
$x\in\RR^d$. The Euclidean open ball centered
at the origin with radius one will be denoted by $B^2$.
We shall abbreviate $B^2$ to $B$ and $|\cdot|_2$ to $|\cdot|$.

\item Let $(X, \mu)$ be a measure space $X$ with a $\sigma$-finite
measure $\mu$.  The space of all measurable functions whose modulus is
integrable with $p$-th power is denoted by $L^p(X)$ for
$p\in(0, \infty)$, whereas $L^{\infty}(X)$ denotes the space of all
measurable essentially bounded functions. The space of all measurable
functions that are weak type $(1, 1)$ will be denoted by
$L^{1, \infty}(X)$.  In our case we will usually have $X=\RR^d$ or
$X=\TT^d$ equipped with Lebesgue measure, and $X=\ZZ^d$ endowed
with counting measure. If $X$ is endowed with counting measure we
will abbreviate $L^p(X)$ to $\ell^p(X)$ and $L^{1, \infty}(X)$ to $\ell^{1, \infty}(X)$.  If the context causes no
confusions we will also abbreviate $\|\cdot\|_{L^p(\RR^d)}$ to
$\|\cdot\|_{L^p}$ and $\|\cdot\|_{\ell^p(\ZZ^d)}$ to
$\|\cdot\|_{\ell^p}$.

\item Let $\calF$ denote the Fourier transform on $\RR^d$ defined for any function 
$f \in L^1\big(\RR^d\big)$ as
\begin{align*}
\calF f(\xi) := \int_{\RR^d} f(x) e^{2\pi i \sprod{\xi}{x}} {\: \rm d}x \quad \text{for any}\quad \xi\in\RR^d.
\end{align*}
If $f \in \ell^1\big(\ZZ^d\big)$ we define the discrete Fourier
transform by setting
\begin{align*}
\hat{f}(\xi) := \sum_{x \in \ZZ^d} f(x) e^{2\pi i \sprod{\xi}{x}} \quad \text{for any}\quad \xi\in\TT^d,
\end{align*}
where $\TT^d$ denotes the $d$-dimensional torus, which will be identified
with $Q:=[-1/2, 1/2]^d$.  To simplify notation we will denote by
$\mathcal F^{-1}$ the inverse Fourier transform on $\RR^d$ or the
inverse Fourier transform (Fourier coefficient) on the torus
$\TT^d$. It will cause no confusions since their meaning will be always clear from the context.

\end{enumerate}

\section{Transference of strong and weak type inequalities: Proof of Theorem \ref{thm:T}}
\label{sec:Tr}

Here we elaborate the arguments from the master thesis of the first author to prove Theorem \ref{thm:T}. The general idea behind the proof of \eqref{eq:16} is as follows. We fix a non-negative bump function $F \colon \RR^d \to \RR$ for which the constant in the corresponding maximal inequality is almost $C(G,p)$. Since dilations are available in the continuous setting, $F$ can be taken to be very slowly varying. Then we sample the values of $F$ at lattice points to produce $f \colon \ZZ^d \to \RR$. Because $F$ is regular, the norms of $F$ and $f$ are almost the same. Moreover, we deduce that $M_*^G F$ cannot be essentially larger than $\mathcal{M}_*^G f$. Indeed, for  $F$ being slowly varying its maximal function is slowly varying as well. Also, given $n \in \ZZ^d$ we see that $M_t^G F(n)$ is certainly not much greater than $f(n)$, unless $t$ is very large. For large values of $t$, in turn, the sets $G_{t} \cap \ZZ^d$ are regular, making the quantities $M_t^G F(n)$ and $\mathcal{M}^G_{t} f(n)$ comparable to each other. The constant in the maximal inequality associated with $f$ is then at least not much smaller than $C(G,p)$. Thus, \eqref{eq:16} holds.

\begin{proof}[Proof of Theorem \ref{thm:T}]
Let $G \subset \RR^d$ be a convex symmetric body. Let $r\in(0, 1)$ and $R>1$ be real numbers such that $B_r\subset G\subset B_R$, where $B_t$ is the Euclidean ball centered at the origin  with radius $t>0$. 
We may assume that $p\in[1, \infty)$, otherwise there is nothing to do. We now distinguish three cases. In the first two cases we prove \eqref{eq:16} for arbitrary $G$ and all $p\in[1, \infty)$. In the third case we show that the equality is attained in \eqref{eq:16} if $G=B^{\infty}$ and $p=1$.
\paragraph{\bf Case 1, when  $p \in (1, \infty)$} Fix
$\eta \in (0,1)$ and take $F \in C_{\rm c}^\infty(\RR^d)$ such that
$F \geq 0$ and
\begin{equation}\label{T1}
\| M^G_* F \|_{L^p(\RR^d)} \geq (1-\eta) C(G,p) \|F\|_{L^p(\RR^d)}.
\end{equation} 
For each $K \in \NN$ let us define $F_K$ by setting
$F_K(x) := F(\frac{x}{K})$. Note that
$\|F_K\|_{L^p(\RR^d)} = K^{d/p} \, \|F\|_{L^p(\RR^d)}$. Moreover,
since $M^G_* F_K (x) = M^G_* F (\frac{x}{K})$, we have
$\| M^G_* F_K \|_{L^p(\RR^d)} = K^{d/p} \, \| M^G_* F \|_{L^p(\RR^d)}$. Next,
we define $f_K \colon \ZZ^d \rightarrow [0, \infty)$ by setting
$f_K(n) := F_K(n)$ for $n \in \ZZ^d$. Then we immediately have
\begin{align}
\label{eq:34}
\|F\|_{L^p(\RR^d)} = \lim_{K \rightarrow \infty} \Big( \frac{1}{K^d}\sum_{n \in \ZZ^d } f_K(n)^p \Big)^{1/p}.
\end{align}
Thus for all sufficiently large $K\in \NN$ (say $K \geq K_1$) we see
\begin{equation}\label{T2}
\| f_K \|_{\ell^p(\ZZ^d)} \leq (1+\eta) K^{d/p} \|F\|_{L^p(\RR^d)}. 
\end{equation}

Choose $N \in \NN$ such that
\begin{equation}\label{T3}
\| M^G_* F \cdot \ind{[-N,N]^d} \|_{L^p(\RR^d)} \geq (1 - \eta) \, \|M^G_* F\|_{L^p(\RR^d)}.
\end{equation}
In a similar way as in \eqref{eq:34}, we conclude that there exists $K_2$ such that for all $K \geq K_2$ we have
\begin{equation}\label{T4}
\Big( \frac{1}{K^d}\sum_{n \in \ZZ^d \cap [-NK, NK]^d} M^G_* F_K(n)^p \Big)^{1/p}
\geq (1 - \eta) \, \| M^G_* F \cdot \ind{[-N,N]^d} \|_{L^p(\RR^d)}.
\end{equation}

Let $\kappa > 0$ be such that $M^G_*F_K(n) \geq \kappa$ for each
$K \in \NN$ and $n \in \ZZ^d \cap [-NK,NK]^d$. We fix
$\varepsilon \in (0, \eta \kappa/2)$ and take $\delta > 0$ for which
$|x - y| < \delta$ implies $|F(x) - F(y)| < \varepsilon$. Since $G_t\subset B_{tR}$, we obtain
\begin{displaymath}
M^G_tF(x) \leq F(x)+\varepsilon, \qquad  t \in (0, \delta R^{-1}),
\end{displaymath}
or, equivalently,
\begin{displaymath}
M^G_tF_K(x) \leq F_K(x)+\varepsilon, \qquad  t \in (0, K \delta R^{-1}).
\end{displaymath}
Our goal is to prove that
\begin{equation}\label{T5}
\calM_*^G f_K(n) \geq (1 - \eta) \, M^G_*F_K(n), \qquad  n \in \ZZ^d \cap [-NK,NK]^d,
\end{equation}
if $K$ is large enough. To this end we shall show separately that
\begin{displaymath}
\calM_*^G f_K(n) \geq M^G_t F_K(n) - \eta \kappa, \qquad t \in (0, K \delta R^{-1}),
\end{displaymath}
and
\begin{displaymath}
\quad \calM_*^G f_K(n) \geq (1 - \eta/2) \, M^G_t F_K(n) - \eta \kappa / 2, \qquad t \geq K \delta R^{-1}.
\end{displaymath} 
	
Fix $n \in \ZZ^d \cap [-NK,NK]^d$. Obviously, if
$t \in (0, K \delta R^{-1})$, then the first inequality follows
\begin{displaymath}
M^G_tF_K(n) \leq F_K(n) + \varepsilon \leq \calM_*^G f_K(n) + \varepsilon \leq \calM_*^G f_K(n) + \eta \kappa.
\end{displaymath}
Hence, we are reduced to prove the second estimate for
$M^G_t F_K(n)$ in the case $t \geq K \delta R^{-1}$. Let $\rho \in (0,1)$ be
such that $|G_{1+2\rho} | \leq (1-\eta/2)^{-1} |G|$ and assume that
$K \geq K_3 := \sqrt{d}R / (r \delta \rho)$. Therefore, for each
$t \geq K \delta R^{-1}$ we have $t + \sqrt{d}/r \leq (1+ \rho)t$. Let $Q_m:=m+B^\infty_{1/2}$
be the cube centered at $m \in \ZZ^d$ and of side length $1$. If $Q_m \cap G_t \neq \emptyset$ for some
$m \in \ZZ^d$, then one can
easily see that
$Q_m \subseteq G_{t+\sqrt{d}/r} \subseteq G_{(1+\rho)t}$, provided
$t \geq K_3 \delta R^{-1}$.  Consequently, we conclude
\begin{align*}
\calM_{t + \sqrt{d}/r}^G f_K(n) & = \frac{1}{|G_{t + \sqrt{d}/r} \cap \ZZ^d|} \, \sum_{m \in G_{t + \sqrt{d}/r} \cap \ZZ^d} f_K(n-m) \\
& \geq  \frac{1}{|G_{t + \sqrt{d}/r} \cap \ZZ^d|} \, \sum_{m \in G_{t + \sqrt{d}/r} \cap \ZZ^d} \Big( \int_{Q_m} F_K(n-x) \, {\rm d}x - \varepsilon \Big) \\
& \geq \Big( \frac{1}{|G_{t + \sqrt{d}/r} \cap \ZZ^d|} \, \int_{G_t} F_K(n-x) \, {\rm d} x \Big) - \varepsilon \\
& \geq \frac{|G_t|}{|G_{(1+\rho)t} \cap \ZZ^d|} M^G_t F_K(n) - \eta \kappa / 2,
\end{align*}
where in the first inequality we have used that
$K \geq K_3 \geq \sqrt{d} / \delta.$ Hence, it remains to show
\begin{displaymath}
|G_t| \geq (1-\eta/2) \ |G_{(1+\rho)t} \cap \ZZ^d|.
\end{displaymath} 
We notice that if $m \in G_{(1+\rho)t} \cap \ZZ^d$, then
$Q_m \subseteq G_{(1+\rho)t + \sqrt{d}/r} \subseteq G_{(1+2\rho)t}$. Thus
we obtain
\begin{displaymath}
|G_{(1+\rho)t} \cap \ZZ^d| \leq |G_{(1+2\rho)t}| \leq (1-\eta/2)^{-1} \, |G_t|.
\end{displaymath}
Finally, combining \eqref{T1}, \eqref{T2}, \eqref{T3}, \eqref{T4}, and
\eqref{T5}, for any $K \geq \max\{K_1, K_2, K_3\}$ we have
\begin{displaymath}
\| \calM^G_* f_K \|_{\ell^p(\ZZ^d)} \geq \frac{(1- \eta)^4}{1 + \eta} \, C(G,p) \, \|f_K \|_{\ell^p(\ZZ^d)}.
\end{displaymath}
Hence, since $\eta \in (0,1)$ was arbitrary, we conclude that
$\calC(G,p) \geq C(G,p)$.
	
\paragraph{\bf Case 2, when $p =1$} The inequality
$\calC(G,1) \geq C(G,1)$ can be deduced in a similar way as it was
done for $p \in (1, \infty)$. We only describe necessary changes. Namely, as in \eqref{T1} we 
fix
$\eta \in (0,1)$ and take $F \in C_{\rm c}^\infty(\RR^d)$ such that
$F \geq 0$ and
\begin{equation}\label{T1'}
\| M^G_* F \|_{L^{1, \infty}(\RR^d)} \geq (1-\eta) C(G,1) \|F\|_{L^1(\RR^d)}.
\end{equation} 
Then we choose $N \in \NN$ such that
\begin{equation}\label{T3'}
\| M^G_* F \cdot \ind{[-N,N]^d} \|_{L^{1, \infty}(\RR^d)} \geq (1 - \eta) \, \|M^G_* F\|_{L^{1, \infty}(\RR^d)}.
\end{equation}
It is easy to see  that for each $x_0, y_0 \in \RR^d$ one has
\begin{displaymath}
| M^G_* F (x_0) - M^G_* F(y_0) | \leq \sup_{|x-y|=|x_0-y_0|} \ |F(x) - F(y)|.
\end{displaymath}
This allows us to deduce for sufficiently large $K\in\NN$ that
\begin{align}
\label{eq:44}
\| M^G_* F_K \cdot \ind{\ZZ^d\cap[-NK,NK]^d} \|_{\ell^{1, \infty}(\ZZ^d)}
\geq (1 - \eta) K^d \, \| M^G_* F \cdot \ind{[-N,N]^d} \|_{L^{1, \infty}(\RR^d)}
\end{align}
with the function $F_K$ as in the previous case. From now on we may
proceed in much the same way as in the previous case to establish
\eqref{T5}. Once \eqref{T5} is proved we combine \eqref{T1'}, \eqref{T2} (with $p=1$), \eqref{T3'}, \eqref{eq:44} and \eqref{T5} to obtain
\begin{displaymath}
\| \calM^G_* f_K \|_{\ell^{1, \infty}(\ZZ^d)} \geq \frac{(1- \eta)^4}{1 + \eta} \, C(G,1) \, \|f_K \|_{\ell^1(\ZZ^d)}
\end{displaymath}
for any $K \geq \max\{K_1, K_2, K_3\}$. Since $\eta \in (0,1)$ was
arbitrary, we conclude that $\calC(G,1) \geq C(G,1)$ as desired. This
completes the first part of Theorem \ref{thm:T}. We now turn our
attention to the case $G = B^\infty$ and show that the last inequality
can be reversed.

\paragraph{\bf Case 3,  when $p =1$ and $G = B^\infty=[-1, 1]^d$}
Given $\eta \in (0,1)$ consider $f \in \ell^1(\ZZ^d)$ and
$\lambda > 0$ such that $f \geq 0$ and
\begin{displaymath}
\lambda \, |\{ n \in \ZZ^d : \calM^{B^\infty}_* f(n) > \lambda \}| \geq (1- \eta) \,  \calC(B^\infty, 1) \, \| f \|_{\ell^1(\ZZ^d)}.
\end{displaymath}
Let $Q_\delta(x):=x+B^\infty_{\delta/2}$ denote
the cube centered at $x\in\RR^d$ and of side length $\delta>0$. For
$\delta\in(0, 1)$, we set
\begin{displaymath}
F_\delta(x) := \sum_{n \in \ZZ^d} f(n) \, |Q_\delta(n)|^{-1}\ind{Q_\delta(n)}(x).
\end{displaymath}
Clearly $\|F_\delta\|_{L^1(\RR^d)} = \| f \|_{\ell^1(\ZZ^d)}$, this is the place where it is essential that we are working with $p=1$.

We show that for each
$n \in \ZZ^d$ one has
\begin{equation}
\label{eq:Ta}
M^{B^\infty}_* F_\delta(x) \geq \calM^{B^\infty}_* f(n), \qquad x \in  Q_{1-\delta}(n).
\end{equation}
To prove \eqref{eq:Ta} we note that
\[
\calM^{B^\infty}_* f(n)=\sup_{t>0}\calM^{B^\infty}_t f(n)=\sup_{N\in\NN_0}\calM^{B^\infty}_N f(n)
\]
since
$|B_t^{\infty}\cap\ZZ^d|=|B_{\lfloor t\rfloor}^{\infty}\cap\ZZ^d|= (2\lfloor t\rfloor+1)^d$,
where $\calM^{B^\infty}_0 f:=f$.

If $N=0$, then for each $n \in \ZZ^d$ and $x \in Q_{1-\delta}(n)$ we
obtain
\begin{align*}
\calM^{B^\infty}_0 f(n)=f(n)= \int_{Q_{1}(x)}F_{\delta}(y)\dif y \le M^{B^\infty}_* F_\delta(x).
\end{align*}

If $N \in \NN$, then $n+ B_N^{\infty}\subseteq x+B_{N+\frac{1-\delta}{2}}^{\infty}$ 
for each $n \in \ZZ^d$ and 
$x \in Q_{1-\delta}(n)$. Therefore, 
\begin{align*}
\calM^{B^\infty}_N f(n)&=
\frac{1}{|n+ B_N^{\infty}\cap\ZZ^d|} \sum_{k\in n+ B_N^{\infty}\cap\ZZ^d}f(k)\\
&\le
\frac{1}{|n+ B_N^{\infty}\cap\ZZ^d|} \sum_{k\in\ZZ^d }\int_{Q_{\delta}(k)}\ind{x+B_{N+\frac{1-\delta}{2}}^{\infty}}(k)f(k) |Q_{\delta}(k)|^{-1}\ind{Q_{\delta}(k)}(y)\dif y\\
&\le \frac{1}{|x+B_{N+1/2}^{\infty}|}\int_{x+B_{N+1/2}^{\infty}} F_{\delta}(y)\,\dif y
\le M^{B^\infty}_* F_\delta(x),
\end{align*}
since $|n+ B_N^{\infty}\cap\ZZ^d|=(2N+1)^d=|x+B_{N+1/2}^{\infty}|$.
Hence \eqref{eq:Ta} follows, and consequently we obtain
\begin{align*}
\lambda \, |\{ x \in \RR^d : M^{B^\infty}_* F_\delta(x) > \lambda \}|
& \geq \lambda \, (1-\delta)^d \, |\{ n \in \ZZ^d : \calM^{B^\infty}_* f(n) > \lambda \}| \\
& \geq (1-\eta) \, (1-\delta)^d \, \calC(B^\infty, 1) \, \|F_\delta \|_{L^1(\RR^d)}.
\end{align*} 
Since $\eta$ and $\delta$ were arbitrary, we conclude that
$C(B^\infty, 1) \geq \calC(B^\infty, 1)$. Finally, combining this
inequality with the previous result we obtain
$C(B^\infty, 1) = \calC(B^\infty, 1)$ and the proof of Theorem \ref{thm:T} is completed.
\end{proof}

\section{Discrete maximal operator for $B^q$ and large scales: Proof of Theorem \ref{thm:1}}
\label{sec:Bq}

The purpose of this section is to prove Theorem \ref{thm:1}. We will follow the ideas from
\cite[Section 5]{BMSW4}. From now on we assume that $q\in(2,\infty)$ is fixed. By $Q := B^\infty_{1/2}$ we mean the unit cube centered at the origin.

\begin{lemma}\label{lem:1}
Let
$\tilde{C}_q := \max \big\{ \sum_{k=1}^\infty \big| {q \choose 2k} \big| \, 2^{-2k}, \big(\frac{3}{2}\big)^q \big\}$. If
$N\geq d^{\frac{1}{2}+\frac{1}{q}}$, then
\begin{equation*}
|B_N^q\cap\ZZ^d|\leq 2 |B_{N_1}^q|\leq 2 e^{\frac{\tilde{C}_q}{q}}|B_N^q|,
\end{equation*}
where $N_1:=N\big(1+ d^{-1} \tilde{C}_q \big)^{\frac{1}{q}}$.
\end{lemma}

\begin{proof}
Since $q > 2$, by Hanner's inequality for $x\in B_N^q\cap\ZZ^d$ and $y\in Q$ we obtain
\begin{equation}\label{eq:Han_ineq}
\norm{x+y}^q_q +\norm{x-y}^q_q\leq (\norm{x}_q+\norm{y}_q)^q
+|\norm{x}_q-\norm{y}_q|^q.
\end{equation}
Moreover, for every $x\in B_N^q$ we have
$|\{y\in Q : \norm{x+y}_q> \norm{x-y}_q\}| = |\{y\in Q : \norm{x+y}_q< \norm{x-y}_q\}|$,
which implies
$|\{y\in Q : \norm{x+y}_q\leq \norm{x-y}_q\}|\geq 1/2$. Hence,
\begin{align}
\label{eq:3}
\begin{split}
|B_N^q\cap \ZZ^d|=\sum_{x\in B_N^q\cap\ZZ^d} 1&\leq 2\sum_{x\in B_N^q\cap\ZZ^d}\int_Q \ind{\{y\in\RR^d : \norm{x+y}_q\leq \norm{x-y}_q\}}(z)\,{\rm d}z\\
&\leq 2\sum_{x\in B_N^q\cap\ZZ^d}\int_Q \ind{\{y\in\RR^d : 2\norm{x+y}_q^q\leq (\norm{x}_q+\norm{y}_q)^q+\abs{\norm{x}_q-\norm{y}_q}^q\}}(z)\,{\rm d}z
\end{split}
\end{align}
	
We shall estimate
$(\norm{x}_q+\norm{y}_q)^q+|\norm{x}_q-\norm{y}_q|^q$ for
$x\in B_N^q\cap\ZZ^d$ and $y\in Q$. Let us firstly assume that
$\norm{x}_q\geq 2 \norm{y}_q$. By Newton's generalized binomial
theorem we have
\begin{align*}
(\norm{x}_q+\norm{y}_q)^q+|\norm{x}_q-\norm{y}_q|^q
&=\sum_{k=0}^\infty {q \choose k} \norm{y}_q^k \norm{x}_q^{q-k}+\sum_{k=0}^\infty {q \choose k} (-1)^k\norm{y}_q^k \norm{x}_q^{q-k}\\
&=2\sum_{k=0}^\infty {q \choose 2k} \norm{y}_q^{2k} \norm{x}_q^{q-2k}\\
& \leq 2 \bigg(   \norm{x}_q^q + \norm{x}_q^{q-2} \norm{y}_q^2 \sum_{k=1}^\infty \bigg| {q \choose 2k} \bigg| \, 2^{-2k+2} \bigg) \\
& \leq 2 \bigg(   N^q + N^{q-2} d^{\frac{2}{q}} \sum_{k=1}^\infty \bigg| {q \choose 2k} \bigg| \, 2^{-2k} \bigg) \\
& \leq 2N^q \bigg( 1 + d^{-1} \sum_{k=1}^\infty \bigg| {q \choose 2k} \bigg| \, 2^{-2k} \bigg).
\end{align*}
On the other hand, if $\norm{x}_q\leq 2\norm{y}_q$, then
\begin{equation*}
(\norm{x}_q+\norm{y}_q)^q+|\norm{x}_q-\norm{y}_q|^q
\leq 2\big(3\norm{y}_q \big)^q\leq 2 \bigg(\frac{3}{2}\bigg)^q d
\leq 2 N^q d^{-\frac{q}{2}} \bigg(\frac{3}{2}\bigg)^q \leq 2 N^q d^{-1} \bigg(\frac{3}{2}\bigg)^q.
\end{equation*}
Combining the above gives
\begin{equation}\label{eq:4}
(\norm{x}_q+\norm{y}_q)^q+|\norm{x}_q-\norm{y}_q|^q\leq 2N^q\big(1+  d^{-1} \tilde{C}_q\big)=2N_1^q.
\end{equation}
Finally, by \eqref{eq:3} and \eqref{eq:4} we obtain
\begin{align*}
\abs{B_N^q\cap \ZZ^d}&\leq 2\sum_{x\in B_N^q\cap\ZZ^d}\int_Q \ind{\{y\in\RR^d : \norm{x+y}_q\leq N(1+d^{-1} \tilde{C}_q )^{1/q}\}}(z)\,{\rm d}z\\
&=2\sum_{x\in B_N^q\cap\ZZ^d}\int_Q \ind{B^q_{N(1+d^{-1}\tilde{C}_q)^{1/q}}}(x+z)\,{\rm d}z\\
&\leq 2|B_{N_1}^q|\\
&=2 \big(1+d^{-1} \tilde{C}_q \big)^{\frac{d}{q}} |B_N^q|\\
&\leq 2 e^{\frac{\tilde{C}_q}{q}}|B_N^q|,
\end{align*}
which finishes the proof.	
\end{proof}

We remark that Lemma \ref{lem:1} holds for $q = 2$ without any changes
with $\tilde{C}_2=9/4,$ so that $e^{\tilde{C}_2/2}=e^{9/8}$.

\begin{lemma}\label{lem:2}
Let $ a >0$. Take $N\geq a d$, $0\leq j\leq N-1$, and $x\in\RR^d$ such
that
\begin{equation}\label{eq:1}
N\bigg(1+\frac{j}{N}\bigg)^{\frac{1}{q}}\leq \norm{x}_q\leq N\bigg(1+\frac{j+1}{N}\bigg)^{\frac{1}{q}}.
\end{equation}
If $d\in\NN$ is sufficiently large (depending only on $ a $ and $q$),
then
\begin{equation*}
\abs{Q\cap(B_N^q-x)}=\abs{\{y\in Q : x+y\in B_N^q\}}\leq e^{-\frac{7}{128q^2} j^2}.
\end{equation*}
\end{lemma}

\begin{proof}
Note that the claim is trivial for $j=0$.  Therefore, let
$1\leq j\leq N-1$ and take $x\in\RR^d$ such that \eqref{eq:1}
holds. Assume that $y\in Q$ and $x+y\in B_N^q$. Then
\begin{equation}
\label{eq:45}
\norm{x}_q^q -\norm{x+y}_q^q \geq N^q\bigg(1+\frac{j}{N}\bigg) -N^q=jN^{q-1}.
\end{equation}
We shall also estimate the expression on the left hand side of the above inequality from above. Let $I_i:=\abs{x_i}^q-\abs{x_i+y_i}^q$, then 
\begin{equation*}
\norm{x}_q^q-\norm{x+y}_q^q=\sum_{i=1}^d I_i.
\end{equation*} 
For $\abs{x_i}>2 \abs{y_i}$ by Newton's generalized binomial theorem
we have
\begin{align*}
I_i=\abs{x_i}^q-\abs{\abs{x_i}+\sgn(x_i)y_i}^q&=\abs{x_i}^q -\abs{\sum_{k=0}^\infty {q \choose k} \abs{x_i}^{q-k} \sgn(x_i)^k y_i^k}\\
&\leq \abs{x_i}^q- \abs{\abs{x_i}^q+q\abs{x_i}^{q-1}\sgn(x_i)y_i}+\sum_{k=2}^\infty \abs{{q \choose k}} \abs{x_i}^{q-k} \abs{y_i}^k\\
&\leq -q\abs{x_i}^{q-1}\sgn(x_i)y_i+\sum_{k=2}^\infty \abs{{q \choose k}} \abs{x_i}^{q-k} \abs{y_i}^k\\
&\leq -q\abs{x_i}^{q-1}\sgn(x_i)y_i+ \abs{x_i}^{q-2} \abs{y_i}^2 \sum_{k=2}^\infty \abs{{q \choose k}} 2^{2-k}.
\end{align*}
On the other hand, if $\abs{x_i}\leq 2\abs{y_i}\le 1$, and $A_q:=1 + \big(\frac{3}{2}\big)^q+q$, then
\begin{align*}
I_i\leq (2 \abs{y_i})^q +(3\abs{y_i})^q\leq 1 + \bigg(\frac{3}{2}\bigg)^q
\leq -q\abs{x_i}^{q-1}\sgn(x_i)y_i+A_q.
\end{align*}
	
Let
$\tilde{x}=\big(\abs{x_1}^{q-1}\sgn(x_1),\ldots,\abs{x_d}^{q-1}\sgn(x_d)\big).$
Combining the above we obtain by H\"{o}lder's inequality
\begin{align*}
\norm{x}_q^q -\norm{x+y}_q^q&\leq -q\langle \tilde{x},y\rangle +A_qd +\sum_{k=2}^\infty \abs{{q \choose k}} 2^{2-k} \sum_{i=1}^d \abs{x_i}^{q-2} \abs{y_i}^2\\
&\leq -q\langle \tilde{x},y\rangle +A_qd +\sum_{k=2}^\infty \abs{{q \choose k}} 2^{2-k} \norm{x}_q^{q-2} \norm{y}_q^{2}\\
&\leq -q\langle \tilde{x},y\rangle +A_qd + N^{q-2}d^{\frac{2}{q}}\sum_{k=2}^\infty \abs{{q \choose k}} 2^{-k+1}\\
&\leq -q\langle \tilde{x},y\rangle +\bigg(A_q + \sum_{k=2}^\infty \abs{{q \choose k}} 2^{-k+1} \bigg) N^{q-2}d^{\frac{2}{q}};
\end{align*}
to ensure the validity of the last inequality above we take $d$ so
large that $ad\ge d^{1/q}.$
	
By \eqref{eq:45} and the previous display, since $q>2$, we obtain
\begin{equation}\label{eq:7}
-q\langle \tilde{x},y\rangle\geq jN^{q-1}-\bigg(A_q + \sum_{k=2}^\infty \abs{{q \choose k}} 2^{-k+1} \bigg) N^{q-2}d^{\frac{2}{q}}
\geq \frac{1}{2} jN^{q-1},
\end{equation}
provided that $d$ is so large that
\begin{equation}
\label{eq:lem:2:1}
\bigg(A_q + \sum_{k=2}^\infty \abs{{q \choose k}} 2^{-k+1} \bigg) a^{-1}d^{-1+\frac{2}{q}}\le \frac12.
\end{equation}
Note that
\begin{equation*}
\norm{\tilde{x}}_2=\norm{x}_{2q-2}^{q-1}\leq \norm{x}_q^{q-1}\leq 2 N^{q-1}.
\end{equation*} 
Hence, for $y\in Q$ and $x+y\in B_N^q$ we obtain
\begin{equation*}
\left\langle -\frac{\tilde{x}}{\norm{\tilde{x}}_2},y \right\rangle\geq \frac{j}{4q}.
\end{equation*}
We know from \cite[Inequality (5.6)]{BMSW4} that for every unit vector
$z\in\RR^d$ and for every $s>0$ we have
\begin{align*}
|\{y\in Q: \langle z, y\rangle\ge
s\}|\le e^{-\frac78s^2}. 
\end{align*}
Applying this inequality  for
$z=-\frac{\tilde{x}}{\norm{\tilde{x}}_2}$ and $s=j/(4q)$ we arrive at
\begin{align*}
\abs{\{y\in Q : x+y\in B_N^q\}}\leq
\Big|\Big\{y\in Q : \bigg\langle -\frac{\tilde{x}}{\norm{\tilde{x}}_2},y \bigg\rangle\geq \frac{j}{4q}\Big\}\Big|
\leq e^{-\frac{7j^2}{128q^2} }.
\end{align*}
This concludes the proof of the lemma.
\end{proof}

A comment is in order here. For $q=2$, a version of Lemma \ref{lem:2} holds for all $d\in \NN$ under the restriction 
\begin{equation}\label{eq:8}
 a  \geq 2 \, \bigg(  1 + \bigg(\frac{3}{2}\bigg)^2+2 + \sum_{k=2}^\infty \abs{{2 \choose k}} 2^{-k+1} \bigg) =2(1+(3/2)^2+2+1/2)=\frac{23}{2}.
\end{equation}
The above ensures that \eqref{eq:lem:2:1} (and hence also \ref{eq:7}) is satisfied for all $d \in \NN$.

\begin{lemma}\label{lem:3}
Let $a>0$. If $d\in\NN$ is sufficiently large (depending on $a$ and
$q$) then there exists a constant $C_q'>0$ depending only on $q$ 
such that for all $N\ge ad$ one has
\begin{equation*}
|B_N^q|\leq C_q' |B_N^q\cap \ZZ^d|.
\end{equation*}
\end{lemma}

\begin{proof}
For each $a>0$ we take $J := J_{q,a} \in \NN$ satisfying
\begin{align}\label{eq:def_J}
\sum_{j\ge J} e^{-\frac{7 j^2}{128q^2}}e^{2(j+1)/(aq)}\le \frac{1}{4 }e^{-\frac{\tilde{C}_q}{q}},
\end{align}
where $\tilde{C}_q>0$ is the constant specified in Lemma \ref{lem:1}.
It suffices to show that for every $M \geq \frac{a d}{2} $ we have
\begin{equation}\label{eq:2}
|B^q_M|\le 2|B^q_{M(1+J/M)^{1/q}}\cap\ZZ^d|.
\end{equation}
Indeed, take $d$ so large that $ad\ge 2 J.$ Then for $N \geq a d$ we
can find $M\in[\frac{ad}{2},N],$ such that $N=M(1+J/M)^{1/q}$
and thus \eqref{eq:2} gives
\begin{align*}
|B^q_N|= (1+J/M)^{d/q}|B^q_M|\leq  2e^{2J/(aq)}|B^q_{N}\cap\ZZ^d|.
\end{align*}
	
Our aim now is to prove \eqref{eq:2}. Define
$U_j=\big\{x\in\RR^d: M\big(1+\frac{j}{M}\big)^{1/q}< \norm{x}_q\leq M\big(1+\frac{(j+1)}{M}\big)^{1/q}\big\}$
for $j \geq 0$ and observe that
\begin{align*}
|B^q_M|
&=\sum_{x\in\ZZ^d}\int_{Q}\ind{B^q_M}(x+y){\rm d}y\\
&=\sum_{x\in B^q_M\cap\ZZ^d}\int_{Q}\ind{B^q_M}(x+y){\rm
d}y+\sum_{0\le j<J}\sum_{x\in U_j\cap\ZZ^d}\int_{Q}\ind{B^q_M}(x+y){\rm d}y\\
&\hspace{4.5cm}	+ \sum_{j\ge J}\sum_{x\in U_j\cap\ZZ^d}\int_{Q}\ind{B^q_M}(x+y){\rm d}y\\
&\le |B^q_{M(1+J/M)^{1/q}}\cap\ZZ^d|+\sum_{j=J}^{M-1}\sum_{x\in U_j\cap\ZZ^d}|Q\cap (B^q_M-x)|.
\end{align*}
Indeed, if $j\ge M$, then $|Q\cap (B^q_M-x)|=0$ holds for each
$x\in U_j$; here we take $d$ so large that
$ad(2^{1/q}-1)\ge d^{1/q}$. Clearly
$M \geq \frac{a d}{2} \geq d^{\frac{1}{2} + \frac{1}{q}},$ if $d$ is
large enough. Now applying Lemma \ref{lem:2} (with $M$ in place of $N$
and $a/2$ in place of $a$) together with Lemma \ref{lem:1} we see that
for sufficiently large $d$ one has
\begin{align*}
\sum_{j=J}^{M-1}\sum_{x\in U_j\cap\ZZ^d}|Q\cap (B^q_M-x)|&\leq \sum_{j=J}^{M-1}\sum_{x\in U_j\cap\ZZ^d} e^{-\frac{7 j^2}{128q^2}}\\
&\leq \sum_{j=J}^{M-1}\big|B^q_{M(1+\frac{j+1}{M})^{1/q}}\cap\ZZ^d\big| e^{-\frac{7 j^2}{128q^2}}\\
&\leq 2e^{\frac{\tilde{C}_q}{q}}  \sum_{j=J}^{\infty}\abs{B^q_{M}}\bigg(1+\frac{j+1}{M}\bigg)^{d/q} e^{-\frac{7 j^2}{128q^2}}\\
&\leq 2e^{\frac{\tilde{C}_q}{q}}  \abs{B^q_{M}}  \sum_{j=J}^{\infty} e^{2(j+1)/(aq)} e^{-\frac{7 j^2}{128q^2}}.
\end{align*}
Hence, by \eqref{eq:def_J} we have
\begin{equation*}
|B^q_M|\leq |B^q_{M(1+J/M)^{1/q}}\cap \ZZ^d|+\frac{1}{2}|B^q_M|,
\end{equation*}
which finishes the proof.
\end{proof}

We make a similar remark as below Lemma \ref{lem:2}. For $q=2$,
the claim of Lemma \ref{lem:3} holds for all $d\in\NN$ provided that the $a>0$
is large enough. In this case it suffices to take
\[
a\geq \max(23,2J_{2,23}),\]
where $J:=J_{2,23}$ is a non-negative integer satisfying
\[
\sum_{j\ge J} e^{-\frac{7 j^2}{128q^2}}e^{2(j+1)/23}\le \frac{1}{4 }e^{-\frac{9}{8}}.
\]
Then the implied constant $C'_2$ is equal to $2 e^{J/23}.$ Finally, in
the proof of Theorem \ref{thm:1} below it suffices to take
$N \geq C'_2 d$. We leave the details to the interested reader.

\begin{proof}[Proof of Theorem \ref{thm:1}]
Fix $p\in(1,\infty)$. It is well known that for any $q\in(2,\infty)$ and any  $d\in\NN$ one has
\begin{equation*}
\big\|\sup_{N >0}|\calM_N^{B^q}f|\big\|_{\ell^p(\ZZ^d)}\leq C_d(p,q) \|f\|_{\ell^p(\ZZ^d)},
\end{equation*}
with a constant $C_d(p,q)$ which depends on the dimension $d.$ Thus we may  assume that the dimension
$d$ is large enough, in fact larger than any fixed number $d_0$ (which may
depend on $q\in (2,\infty)$ and $a>0$).
	
Let $f \colon \ZZ^d\to \CC$ be a non-negative function. Define
$F \colon \RR^d\to \CC$ by setting
\[
F(x):=\sum_{y\in\ZZ^d}f(y)\ind{y+Q}(x).
\]
Clearly $\|F\|_{L^p(\RR^d)}=\|f\|_{\ell^p(\ZZ^d)}$.
	
Fix $a > 0$ and let $\tilde{C}_q>0$ be the constant specified in Lemma
\ref{lem:1}. Take $N\ge a d$ and define
$N_1:=N\big(1+d^{-1} \tilde{C}_q \big)^{1/q}$.  Observe that
$ad \geq d^{\frac{1}{2}+\frac{1}{q}}$ when $d$ is large enough.
Hence, by \eqref{eq:4} and \eqref{eq:Han_ineq}, for $z\in Q$
and $y\in B^q_{N}$ we have the estimate
\[
\norm{y+z}_q\le N_1
\]
on the set $\{z\in Q : \norm{y+z}_q\le \norm{y-z}_q\}$, and the
Lebesgue measure of this set is at least $1/2$. Then by Lemma \ref{lem:3} for
sufficiently large dimension $d$ and all $x\in\ZZ^d$ we obtain
\begin{align}
\label{eq:5}
\begin{split}
\mathcal M_N^{B^q}f(x)&
=\frac1{|B^q_N\cap\ZZ^d|}\sum_{y\in B^q_N\cap\ZZ^d}f(x+y)\ind{B^q_N}(y)\\
&=\frac1{|B^q_N\cap\ZZ^d|}\sum_{y\in B^q_N\cap\ZZ^d}f(x+y)\int_Q \ind{B^q_N}(y)\,dz\\
&\lesssim_q \frac1{|B^q_N|}\sum_{y\in\ZZ^d}f(x+y)\int_{Q}\ind{B^q_{N_1}}(y+z) \, {\rm d}z\\
&=\frac1{|B^q_N|}\sum_{y\in\ZZ^d}f(y)\int_{x+B^q_{N_1}}\ind{y+Q}(z) \, {\rm d}z\\
&=\frac1{|B^q_N|}\int_{x+B^q_{N_1}}F(z) \, {\rm d}z\\
&=\bigg(\frac{N_1}{N}\bigg)^d\frac{1}{|B^q_{N_1}|}\int_{B^q_{N_1}}F(x+z) \, {\rm d}z\\
&\lesssim_q\frac{1}{|B^q_{N_1}|}\int_{B^q_{N_1}}F(x+z) \, {\rm d}z\\
&=M_{N_1}^{B^q}F(x).
\end{split}
\end{align}

Let us now take $N_2:=N_1\big(1+d^{-1} \tilde{C}_q \big)^{1/q}.$
Similarly as above, for $y\in Q$ and $z\in B^q_{N_1}$ we have
\begin{equation*}
\norm{y+z}_q\leq N_2
\end{equation*}
on the set $\{y\in Q : \norm{y+z}_q\le \norm{y-z}_q\}$, and the
Lebesgue measure of this set is at least $1/2$. Therefore, Fubini's theorem leads to
\begin{align}
\label{eq:6}
\begin{split}
M_{N_1}^{B^q}F(x)&=\frac{1}{|B^q_{N_1}|}\int_{B^q_{N_1}}F(x+z) \, {\rm d}z\\
&\leq \frac2{|B^q_{N_1}|}
\int_{\RR^d}F(x+z)\ind{B^q_{N_1}}(z)\int_{Q}\ind{B^q_{N_2}}(z+y) \, {\rm d}y{\rm d}z\\
&=2\Big(\frac{N_2}{N_1}\Big)^d \int_{Q} \frac{1}{|B^q_{N_2}|}
\int_{\RR^d}F(x+z-y)\ind{B^q_{N_1}}(z-y)\ind{B^q_{N_2}}(z) \, {\rm d}z{\rm d}y\\
&\lesssim \int_{x+Q}\frac1{|B^q_{N_2}|}\int_{\RR^d}F(z-y)\ind{B^q_{N_2}}(z) \, {\rm d}z{\rm d}y\\
&= \int_{x+Q}M_{N_2}^{B^q}F(y) \, {\rm d}y.
\end{split}
\end{align}
Denote $C_{d,q} := a \, \big(1+d^{-1}\tilde{C}_q\big)^{2/q} d.$
Combining \eqref{eq:5} with \eqref{eq:6} and applying H\"{o}lder's
inequality, we obtain
\begin{align*}
\big\|\sup_{N\ge a d}|\mathcal M_N^{B^q}f|\big\|_{\ell^p(\ZZ^d)}^p
&\lesssim_q\sum_{x\in\ZZ^d} \Big(\int_{x+Q} \sup_{N\ge
C_{d,q}}|M_N^{B^q}F(y)|\Big)^p  {\rm d}y\\
&\leq \sum_{x\in\ZZ^d} \int_{x+Q}\sup_{N\ge
C_{d,q}}\big| M_N^{B^q}F(y)\big|^p  {\rm d}y \\
&=\big\|\sup_{N\ge C_{d,q}}\big|M_N^{B^q}F\big|\big\|_{L^p(\RR^d)}^p.
\end{align*}
By the dimension-free $L^p(\RR^d)$ boundedness of the maximal operator $M_*^{B^q}$ (proved in \cite{Mul1}), we obtain 
\begin{equation*}
\big\|\sup_{N\ge C_{d,q}}\big|M_N^{B^q}F\big|\big\|_{L^p(\RR^d)}^p\lesssim_q\|F\|_{L^p(\RR^d)}^p=\|f\|_{\ell^p(\ZZ^d)}^p.
\end{equation*}
This proves Theorem \ref{thm:1}.
\end{proof}

\section{Decrease dimension trick: Proof of Theorem \ref{thm:10}}
\label{sec:4}
We now prove Theorem \ref{thm:10} by adapting the methods
introduced in \cite[Section 2]{BMSW2} to the case of $q$-balls. In
fact this section is a technical elaboration of \cite[Section
2]{BMSW2}. However, we have decided to
provide necessary details due to intricate technicalities. Throughout this section we will abbreviate
$\|\cdot\|_{\ell^p(\ZZ^d)}$ to $\|\cdot\|_{\ell^p}$ and
$\|\cdot\|_{L^p(\RR^d)}$ to $\|\cdot\|_{L^p}$. We fix
$q \in [2, \infty)$ and recall that $\mathcal M_N^{B^q}$ is the
operator whose multiplier is given by
$$
\mathfrak m^{B^q}_N(\xi):
=\frac{1}{|B^q_N\cap\ZZ^d|}
\sum_{x\in B^q_N\cap\ZZ^d}e^{2\pi i \xi\cdot x}, \qquad \xi\in\TT^d\equiv[-1/2, 1/2)^d.
$$
For each $\xi\in\TT^d$ we will write
$\|\xi\|^2:=\|\xi_1\|^2+\ldots+\|\xi_d\|^2$, where
$\|\xi_j\|=\dist(\xi_j, \ZZ)$ for any $j\in\NN_d$. Since we identify
$\TT^d$ with $[-1/2, 1/2)^d$, it is easy to see that the norm
$\|\cdot\|$ coincides with the Euclidean norm $|\cdot|_2$ restricted
to $[-1/2, 1/2)^d$. It is also very well known that $\|\eta\|\simeq|\sin(\pi\eta)|$
for every $\eta\in\TT$, since
$|\sin(\pi\eta)|=\sin(\pi\|\eta\|)$ and for $0\le|\eta|\le 1/2$ we
have
\begin{align}
\label{eq:103}
2|\eta|\le|\sin(\pi\eta)|\le \pi|\eta|.
\end{align}

The proof of Theorem \ref{thm:10} is based on Proposition
\ref{prop:1} and Proposition \ref{prop:2}, which give respectively estimates of
the multiplier $\mathfrak m^{B^q}_N(\xi)$ at the origin and at
infinity. These estimates will be described in terms of the
proportionality constant
\begin{align*}
\kappa(d, N):= \kappa_q(d, N):=Nd^{-1/q}.
\end{align*}

\begin{proposition}
	\label{prop:1}
	For every $d, N \in \NN$ and for every $\xi \in \TT^d$ we have
\begin{align}
\label{eq:22}
|\mm^{B^q}_N(\xi) - 1| \leq 2 \pi^2 \kappa(d,N)^2 \| \xi \|^2. 
\end{align}
\end{proposition}

\begin{proposition}
	\label{prop:2}
	There is a constant $C_q>0$ such that for any $d, N\in\NN$ if $10\le \kappa(d, N) \leq 50qd^{1-1/q}$, then for all $\xi\in\TT^d$ we have
	\begin{align}
		\label{eq:23}
		|\mm^{B^q}_N(\xi)|\le C_q \big((\kappa(d, N)\|\xi\|)^{-1}+\kappa(d, N)^{-\frac{1}{7}}\big).
	\end{align}
\end{proposition}

Before we prove Proposition \ref{prop:1} and Proposition \ref{prop:2} we show how 
\eqref{eq:20} follows from \eqref{eq:22} and
\eqref{eq:23}.

\begin{proof}[Proof of Theorem \ref{thm:10}]
Since $\mathbb D_{C_1, C_2}$ is a subset of the dyadic set $\mathbb D$
we can assume, without loss of generality, that $C_1=C_2=10$. For every $t>0$ let $P_t$ be the
semigroup with the multiplier
\begin{align*}
\mathfrak p_t(\xi):=e^{-t\sum_{i=1}^d\sin^2(\pi\xi_i)}, \qquad \xi\in\TT^d.
\end{align*}
It follows from \cite{Ste1} (see also \cite{BMSW3} for more details)
that for every $p\in(1, \infty)$ there is $C_p>0$ independent of
$d\in\NN$ such that for every $f\in\ell^p(\ZZ^d)$ we have
\begin{align*}
\Big\|\sup_{t>0}|P_tf|\Big\|_{\ell^p}\le C_p \|f\|_{\ell^p}.
\end{align*}
It suffices to compare the averages $\mathcal M^{B^q}_N$ with
$P_{N^2/d^{2/q}}$. Namely, the proof of \eqref{eq:20} will be
completed if we obtain the dimension-free estimate on $\ell^2(\ZZ^d)$
for the following square function
\[
Sf(x):=\Big(\sum_{N\in \mathbb D_{C_1, C_2}}|\mathcal M_Nf(x)-P_{N^2/d^{2/q}}f(x)|^2\Big)^{1/2}, \qquad x\in\ZZ^d.
\]
By Plancherel's formula, \eqref{eq:22}, and \eqref{eq:23}, we can
estimate $\|S(f)\|_{\ell^2}^2$ by
\begin{align*}
C_q' \int_{\TT^d}\,
\bigg(\sum_{\substack{m\in\ZZ:\\10d^{1/q}\le 2^m\le 10d}}
\min\bigg\{\frac{2^{2m}}{d^{2/q}}\|\xi\|^2,\bigg(\frac{2^{2m}}{d^{2/q}}\|\xi\|^2\bigg)^{-1}\bigg\}+d^{2/7q}\sum_{\substack{m\in\ZZ:\\10d^{1/q}\le 2^m\le 10d}}
2^{-2m/7}\bigg) |\hat{f}(\xi)|^2{\rm d}\xi,
\end{align*}
where $ C_q'$ is a constant that depends only on $q.$ To complete the
proof we note that the integral above is clearly bounded by
$C \|f\|_{\ell^2}^2$ for a suitable constant $C>0$ independent of $d$.
\end{proof}

The rest of this section is devoted to proving Proposition
\ref{prop:1} and Proposition \ref{prop:2}. We emphasize that in the proof of Proposition \ref{prop:1} the assumption $q\geq 2$ is crucial.

\begin{proof}[Proof of Proposition \ref{prop:1}]
Since the balls $B_N^q$ are symmetric under permutations and sign
changes we may repeat the proof of \cite[Proposition 2.1]{BMSW2}
reaching
$$|\mm^{B^q}_N(\xi) - 1| \le 	\frac{2}{|B^q_N \cap \ZZ^d |} \sum_{j=1}^d \sin^2(\pi \xi_j) \sum_{x \in B^q_N \cap \ZZ^d} \frac{|x|_2^2}{d}.$$
Observe that $|x|_2^2\le |x|_q^2\cdot  d^{1-2/q}$. Indeed, for $q=2$ this is simply an equality, and for $q>2$ it suffices to apply H\"older's inequality for the pair $(\frac{q}{2},\frac{q}{q-2})$. Consequently,
\begin{displaymath}
|\mm^{B^q}_N(\xi) - 1|  \leq \frac{2}{|B^q_N \cap \ZZ^d |}  \sum_{j=1}^d \sin^2(\pi \xi_j)
\sum_{x \in B^q_N \cap \ZZ^d} \frac{|x|_q^2}{d^{2/q}} \leq 2 \pi^2 \kappa(d,N)^2 \|\xi\|^2,
\end{displaymath}
which gives the claim.
\end{proof}

The proof of Proposition \ref{prop:2} can be deduced from a series of
auxiliary lemmas, which we formulate and prove below.  In what follows
for an integer $1\le r\le d$ and a radius $R>0$ we let
$$B_{R}^{q,(r)}=\{x\in\RR^r: \vert x\vert_q\le R \}\qquad\textrm{and}\qquad  S_{R}^{q,(r)}=\{x\in\RR^r:|x|_q=R \}$$
be the $r$-dimensional ball and sphere of radius $R>0,$ respectively.

\begin{lemma}
\label{lem:4}
For all $d, N \in \NN$ we have
\begin{displaymath}
(2 \lfloor \kappa(d,N) \rfloor + 1)^d \leq |B^q_N \cap \ZZ^d|
\leq |B^q_{N + d^{1/q}}|
= \frac{2^d \Gamma(1+\frac{1}{q})^d}{\Gamma(1 + \frac{d}{q})} \Big(N + d^{1/q} \Big)^d.
\end{displaymath}
\end{lemma}

\begin{proof}
The lower bound follows from the inclusion $[-\kappa(d, N), \kappa(d, N)]^d\cap\ZZ^d\subseteq B^q_N\cap\ZZ^d$, while the upper bound is a simple consequence of the triangle inequality. 
\end{proof}

\begin{lemma}
\label{lem:5}
Given $\varepsilon_1, \varepsilon_2\in(0, 1]$ we define for every
$d, N\in\NN$ the set
\[
E=\big\{x\in B^q_N\cap\ZZ^d : |\{i\in\NN_d : |x_i|\ge \varepsilon_2\kappa(d, N)\}|\le\varepsilon_1 d\big\}.
\]
If $\varepsilon_1, \varepsilon_2\in(0, 1/(10q)]$ and
$\kappa(d,N)\ge10$, then we have
\begin{align*}
|E| \le 2e^{-\frac{d}{10}}|B^q_N\cap\ZZ^d|.
\end{align*}
\end{lemma}

\begin{proof}
As in \cite[Lemma 2.4]{BMSW2} we can estimate
\begin{equation}
\label{eq:9}
|E| \le (2\varepsilon_2\kappa(d, N)+1)^{d}+ \sum_{1\le m\le \varepsilon_1d}{{d}\choose{m}}(2\varepsilon_2\kappa(d,N)+1)^{d-m}\big|B_N^{q,(m)}\cap\ZZ^{m}\big|.
\end{equation}
Since $\kappa(d, N) \geq 10$ and $\varepsilon_2 \leq 1/10$, using the lower bound from Lemma \ref{lem:4} gives
\begin{equation}
\label{eq:10}
(2\varepsilon_2\kappa(d, N)+1)^{d} \leq e^{-\frac{16d}{19}} |B^q_N\cap\ZZ^{d}|
\end{equation}
as in the case $q=2$. On the other hand, the upper bound from Lemma \ref{lem:4} can be applied to estimate the sum appearing in \eqref{eq:9} by
\begin{align}
\label{eq:11}
\sum_{1\le m\le \varepsilon_1d}  \frac{d^m}{m!}(2\varepsilon_2\kappa(d,N)+1)^{d-m} \frac{2^m}{\Gamma(1+\frac{m}{q})} d^{m/q} \big( \kappa(d,N) + 1 \big)^m.
\end{align}

Now observe that
\begin{equation}
\label{eq:12}
\frac{1}{\Gamma(1+\frac{m}{q})} \leq \frac{4q e^{m/q}}{(m/(2q))^{-1+m/q}}.
\end{equation} 
Indeed, if $m/q \geq 2$, then

\begin{displaymath}
\frac{1}{\Gamma(1+\frac{m}{q})} \leq \frac{1}{\lfloor \frac{m}{q} \rfloor !} \leq
\frac{e^{\lfloor \frac{m}{q} \rfloor}}{\lfloor \frac{m}{q} \rfloor^{\lfloor \frac{m}{q} \rfloor}} \leq \frac{e^{\frac{m}{q}}}{(m/(2q))^{-1+m/q}},
\end{displaymath}
where in the second inequality we have used that $\frac{1}{n!} \leq \frac{e^n}{n^n}$ holds for any $n \in \NN$. If $m/q \leq 2$, then
\begin{displaymath}
\frac{1}{\Gamma(1+\frac{m}{q})} \leq 2 \leq \frac{4qm}{2q} (2e)^{m/q} (m/q)^{-m/q} = \frac{4q e^{\frac{m}{q}}}{(m/(2q))^{-1+m/q}}.
\end{displaymath}
In the first inequality we used the fact that the gamma function is estimated from below by $1/2$ on the interval $[1,3]$. 

Applying \eqref{eq:12} we get
\begin{align*}
\frac{d^{m+m/q} 2^m}{m! \, \Gamma(1 + \frac{m}{q})} \leq \frac{d^{m+m/q} 2^m e^m 4q e^{m/q}}{m^m (m/(2q))^{-1+m/q}} = \Big( \frac{2de}{m}\Big)^{m(1+1/q)} 2m q^{m/q} \leq \Big( \frac{2deq}{m}\Big)^{m(1+1/q)}.
\end{align*}
Combining this with \eqref{eq:9}, \eqref{eq:10}, and \eqref{eq:11}, and repeating the argument used in \cite[(2.16)]{BMSW2}, we arrive at
\begin{align}
\label{eq:14}
\begin{split}
|E| & \leq  e^{-\frac{16d}{19}} |B^q_N\cap\ZZ^{d}| + \sum_{m=1}^{\lfloor \varepsilon_1 d \rfloor} \Big( \frac{2deq}{m}\Big)^{m(1+1/q)} (2\varepsilon_2\kappa(d,N)+1)^{d-m} \big( \kappa(d,N) + 1 \big)^m \\
&\leq \Big(e^{-\frac{16d}{19}}+\sum_{m=1}^{\lfloor \varepsilon_1 d \rfloor} \Big( \frac{2deq}{m}\Big)^{m(1+1/q)}\Big(\frac{2\varepsilon_2\kappa(d,N)+1}{2 \lfloor \kappa(d,N) \rfloor + 1}\Big)^{d-m}\Big)|B^q_N\cap\ZZ^{d}| \\
& \leq \Big( e^{-\frac{16d}{19}} + e^{- \frac{72d}{95}} \sum_{m=1}^{\lfloor \varepsilon_1 d \rfloor} e^{ \varphi(m) } \Big) |B^q_N\cap\ZZ^{d}|,
\end{split}
\end{align}
where $\varphi(x):= (1+1/q) x \log(\frac{2eqd}{x})$, $x\ge0.$

For $x\in [0, d/(10q)]$ we
have
$$\varphi'(x)=(1+1/q)\log\bigg(\frac{2eqd}{x}\bigg)-(1+1/q)\ge  \log\bigg(\frac{2qd}{x}\bigg)\ge \log 3.$$
Hence, $\varphi$ is increasing on
$[0,\lfloor \varepsilon_1 d \rfloor],$ and arguing as in the proof of
\cite[Lemma 2.4]{BMSW2} we get
\begin{equation*}
\sum_{m=1}^{\lfloor \varepsilon_1 d \rfloor} e^{ \varphi(m) } \leq e^{\varphi(\frac{d}{10q})} \sum_{m=1}^{\lfloor \varepsilon_1 d \rfloor}e^{-(\lfloor \varepsilon_1 d \rfloor-m)\log 3}\le \frac{3}{2} e^{\varphi(\frac{d}{10q})} \leq \frac{3}{2} (20eq^2)^{d/(5q)} \le \frac{3}{2} e^{\frac{2d}{5e}} e^{\frac{4d}{5q}} < \frac{3}{2} e^{\frac{3}{5}d},
\end{equation*}
since $q^{1/q} \leq e^{1/e}$ and $20 < e^3$. Then \eqref{eq:14} gives 
\begin{displaymath}
\frac{|E|}{|B^q_N\cap\ZZ^{d}|} \leq e^{-\frac{16d}{19}}
+e^{-\frac{72d}{95}}\sum_{m=1}^{\lfloor \varepsilon_1 d \rfloor}
e^{\varphi(m)}
\le e^{-\frac{16d}{19}}+ \frac{3}{2}e^{-\frac{3d}{19}}
\le e^{-\frac{d}{10}}\Big(e^{-\frac{141}{190}}+\frac{3}{2}\Big)
\le2e^{-\frac{d}{10}},
\end{displaymath} 
which finishes the proof. 
\end{proof}

Recall that ${\rm Sym}(d)$ denotes the permutation group on
$\NN_d$. We will also write
$\sigma\cdot x=(x_{\sigma(1)}, \ldots, x_{\sigma(d)})$ for every
$x\in\RR^d$ and $\sigma\in{\rm Sym}(d)$. Let $\mathbb P$ be the
uniform distribution on the symmetry group ${\rm Sym}(d)$,
i.e. $\mathbb P(A)={|A|}/{d!}$ for any $A\subseteq{\rm Sym}(d)$, since
$|{\rm Sym}(d)|=d!$. The expectation $\mathbb E$ will be always taken
with respect to the uniform distribution $\mathbb P$ on the symmetry
group ${\rm Sym}(d)$. We will need two lemmas from \cite{BMSW2}.

\begin{lemma}
\label{lem:6}
Assume that $I, J\subseteq\NN_d$ and $|J|=r$ for some
$0\le r\le d$. Then
\begin{align*}
\mathbb P[\{\sigma\in{\rm Sym}(d) : |\sigma(I)\cap J|\le {r|I|}/{(5d)}\}]\le  e^{-\frac{r|I|}{10d}}.
\end{align*}
In particular, if $\delta_1, \delta_2\in(0, 1]$ satisfy
$5\delta_2\le\delta_1$ and $\delta_1d\le |I|\le d$, then we have
\begin{align*}
\mathbb P[\{\sigma\in{\rm Sym}(d) : |\sigma(I)\cap
J|\le \delta_2 r\}]\le e^{-\frac{\delta_1r}{10}}.
\end{align*}  
\end{lemma}

\begin{lemma}
\label{lem:7}
Assume that we have a finite decreasing sequence
$0\le u_d\le\ldots\le u_2\le u_1\le(1-\delta_0)/2$ for some
$\delta_0\in(0, 1)$. Suppose that $I\subseteq\NN_d$ satisfies
$\delta_1d\le |I|\le d$ for some $\delta_1\in(0, 1]$. Then for every
$J=(d_0, d]\cap\ZZ$ with $0\le d_0\le d$ we have
\begin{align*}
\mathbb E\Big[\exp\Big({-\sum_{j\in\sigma(I)\cap J}u_j}\Big)\Big]
\le 3\exp\Big({-\frac{\delta_0\delta_1}{20}\sum_{j\in J}u_j}\Big).
\end{align*}
\end{lemma}

\begin{lemma}
\label{lem:8}
For $d, N\in\NN$, $\varepsilon\in (0, 1/(50q)]$ and an integer
$1\le r\le d$ we define
\[
E=\{x\in B^q_N\cap\ZZ^d : \sum_{i=1}^rx_i^q<\varepsilon^{q+1}\kappa(d, N)^q r\}.
\]
If $\kappa(d, N)\ge10$, then we have
\begin{align}
\label{eq:37}
|E|\le 4e^{-\frac{\varepsilon r}{10}}|B^q_N\cap\ZZ^d|.
\end{align}
As a consequence, $B^q_N\cap\ZZ^d$ can be written as a disjoint sum
\begin{align}
\label{eq:38}
\begin{split}
B^q_N\cap\ZZ^d &= \Big( \bigcup_{\varepsilon^{q+1}\kappa(d, N)^q r\leq l\le N^q}\big(B_{l^{1/q}}^{q,(r)}\cap\ZZ^r\big) \times\big(S^{q,d-r}_{(N^q-l)^{1/q}}\cap\ZZ^{d-r}\big) \Big) \cup E',
\end{split}
\end{align}
where $E' \subset \ZZ^d$ satisfies $|E'|\le 4e^{-\frac{\varepsilon r}{10}}|B^q_N\cap\ZZ^d|$.
\end{lemma}
\begin{proof}
Let $\delta_1\in(0, 1/(10q)]$ be such that $\delta_1\ge5\varepsilon$,
and define $I_x=\{i\in\NN_d: |x_i|\ge\varepsilon\kappa(d, N)\}$. We
have $E\subseteq E_1\cup E_2$, where
\begin{align*}
E_1&=\{x\in B^q_N\cap\ZZ^d : \sum_{i\in I_x\cap\NN_r}|x_i|^q<\varepsilon^{q+1}\kappa(d, N)^q r\ \text{ and }\ |I_x|\ge\delta_1d\},\\
E_2&=\{x\in B^q_N\cap\ZZ^d : |I_x|<\delta_1d\}.
\end{align*}
By Lemma \ref{lem:5} (with $\varepsilon_1=\delta_1$ and
$\varepsilon_2=\varepsilon$) we have
$|E_2|\le2e^{-\frac{d}{10}}|B^q_N\cap\ZZ^d|$, provided that
$\kappa(d, N)\ge10$. Observe that
\begin{align*}
|E_1|&=\sum_{x\in B^q_N\cap\ZZ^d}\frac{1}{d!}\sum_{\sigma\in{\rm Sym}(d)}\ind{E_1}(\sigma^{-1}\cdot x)\\
&=\sum_{x\in B^q_N\cap\ZZ^d}\mathbb P[\{\sigma\in{\rm Sym}(d) :
\sum_{i\in \sigma(I_x)\cap\NN_r}|x_{\sigma^{-1}(i)}|^q<\varepsilon^{q+1}\kappa(d, N)^q r\ \text{ and }\ |\sigma(I_x)|\ge\delta_1d\}],
\end{align*}
since $I_{\sigma^{-1}\cdot x}=\sigma(I_x)$.  Now by Lemma \ref{lem:6}
(with $J=\NN_r$, $\delta_2=\frac{\delta_1}{5}$ and $\delta_1$ as
above) we obtain, for every $x\in B^q_N\cap\ZZ^d$ such that
$|I_x|\ge \delta_1 d$ , the estimate
\begin{align*}
\mathbb P[\{\sigma\in&{\rm Sym}(d) :
\sum_{i\in \sigma(I_x)\cap\NN_r}|x_{\sigma^{-1}(i)}|^q<\varepsilon^{q+1}\kappa(d, N)^q r\ \text{ and }\ |\sigma(I_x)|\ge\delta_1d\}]\\
&\le
\mathbb P[\{\sigma\in{\rm Sym}(d)
: |\sigma(I_x)\cap\NN_r|\le\delta_2r\}]\le 2e^{-\frac{\delta_1 r}{10}},
\end{align*}
since
\[
\{\sigma\in{\rm Sym}(d) : \sum_{i\in \sigma(I_x)\cap\NN_r}|x_{\sigma^{-1}(i)}|^q<\varepsilon^{q+1}\kappa(d, N)^q r\ \text{
and }\ |\sigma(I_x)|\ge\delta_1d \ \text{ and
}\ |\sigma(I_x)\cap\NN_r|>\delta_2r\}=\emptyset.
\]
Thus $|E_1|\le 2e^{-\frac{\varepsilon r}{2}}|B^q_N\cap \ZZ^d|$, which
proves \eqref{eq:37}. To prove \eqref{eq:38} we write
\[
B^q_N\cap\ZZ^d=\bigcup_{l=0}^{N^q}\big(B_{l^{1/q}}^{q,(r)}\cap\ZZ^r\big) \times \big(S^{q,d-r}_{(N^q-l)^{1/q}}\cap\ZZ^{d-r}\big).
\]
Then we see that
\[
\Big(\bigcup_{l=0}^{N^q}\big(B_{l^{1/q}}^{q,(r)}\cap\ZZ^r\big) \times\big(S^{q,d-r}_{(N^q-l)^{1/q}}\cap\ZZ^{d-r}\big)\Big)\cap E^{\bf c} =\Big(\bigcup_{l\ge\varepsilon^{q+1}\kappa(d, N)^q r}\big(B_{l^{1/q}}^{q,(r)}\cap\ZZ^r\big) \times\big(S^{d-r}_{(N^q-l)^{1/q}}\cap\ZZ^{d-r}\big)\Big)\cap E^{\bf c},
\]
and consequently we obtain \eqref{eq:38} with some $E'\subseteq E$.
The proof is completed.
\end{proof}

\begin{lemma}
\label{lem:9}
For $d, N\in\NN$ and $\varepsilon\in(0, 1/(50q)]$ if
$\kappa(d, N)\ge10$, then for every $1\le r\le d$ and $\xi\in\TT^d$ we
have
\begin{align*}
|\mathfrak m^{B^q}_N(\xi)|\le\sup_{l\ge\varepsilon^{q+1}\kappa(d, N)^q r}|\mathfrak
m_{l^{1/q}}^{B^q,(r)}(\xi_1,\ldots, \xi_r)|+4e^{-\frac{\varepsilon r}{10}},
\end{align*}
where
\begin{align*}
\mathfrak m_{R}^{B^q,(r)}(\eta):=\frac{1}{|B_{R}^{q,(r)}\cap\ZZ^d|}\sum_{x\in
B_{R}^{q,(r)}\cap\ZZ^d}e^{2\pi i \eta\cdot x},\qquad \eta\in \TT^r,
\end{align*}
is the lower dimensional multiplier with  $r\in \NN$ and $R>0$.

\end{lemma}
\begin{proof}
	We identify $\RR^d\equiv\RR^r\times\RR^{d-r}$ and $\TT^d\equiv\TT^r\times\TT^{d-r}$ and we will write $\RR^d\ni x=(x^1, x^2)\in \RR^r\times\RR^{d-r}$ and $\TT^d\ni \xi=(\xi^1, \xi^2)\in \TT^r\times\TT^{d-r}$ respectively.
	Invoking \eqref{eq:38} one obtains
	\begin{align*}
	&|\mathfrak m^{B^q}_N(\xi)|\\
	&\le\frac{1}{|B^q_N\cap\ZZ^d|}\sum_{l\ge\varepsilon^{q+1}\kappa(d, N)^q r}\;\sum_{x^2\in
		S_{(N^q-l)^{1/q}}^{q,d-r}\cap\ZZ^{d-r}}|B_{l^{1/q}}^{q,(r)}\cap\ZZ^r|\frac{1}{|B_{l^{1/q}}^{q,(r)}\cap\ZZ^r|}\Big|\sum_{x^1\in
		B_{l^{1/q}}^{q,(r)}\cap\ZZ^r}e^{2\pi i \xi^1\cdot
		x^1}\Big|+4e^{-\frac{\varepsilon r}{10}}\\
	&  \le\sup_{l\ge\varepsilon^{q+1}\kappa(d, N)^q r}|\mathfrak
	m_{l^{1/q}}^{B^q,(r)}(\xi_1,\ldots, \xi_r)|+4e^{-\frac{\varepsilon r}{10}}.
	\end{align*}
	In the last inequality the disjointness in the decomposition from \eqref{eq:38} has been used.
\end{proof}

The next two lemmas give information on the size difference between the balls $B_R^{q,(r)}$ and their shifts $z+B_R^{q,(r)}$ for  $z\in \RR^r$.
\begin{lemma}
\label{lem:10}
Let $R\ge1$ and let $r\in\NN$ be such that $r\le R^{\delta}$ for some
$\delta\in(0, q/(q+1))$. Then for every $z\in\RR^r$ we have
\begin{align}
\label{eq:40}
\big||(z+B_R^{q,(r)})\cap\ZZ^r|-|B_R^{q,(r)}|\big|\le |B_R^{q,(r)}|r^{(q+1)/q}R^{-1}e^{r^{(q+1)/q}/R}\le e|B_R^{q,(r)}|R^{-1+(q+1)\delta/q}.
\end{align}	
\end{lemma}
\begin{proof}
For the proof we refer to \cite[Lemma 2.9]{BMSW2}.
\end{proof}

\begin{lemma}
\label{lem:11}
Let $R\ge1$ and let $r\in\NN$ be such that $r\le R^{\delta}$ for some
$\delta\in(0, q/(q+1))$. Then for every $z\in\RR^r$ we have
\begin{align*}
\big|\big(B_{R}^{q,(r)}\cap\ZZ^r\big)\triangle\big((z+B_{R}^{q,(r)})\cap\ZZ^r\big)\big|&\le
4e\big(r|z|R^{-1}e^{r|z|R^{-1}}+e^{r|z|R^{-1}}R^{-1+(q+1)\delta/q}\big)|B_R^{q,(r)}|\\
&\le
4e\big(|z|R^{-1+\delta}e^{|z|R^{-1+\delta}}+
e^{|z|R^{-1+\delta}}R^{-1+(q+1)\delta/q}\big)|B_R^{q,(r)}|.
\end{align*}
\end{lemma}

\begin{proof}
For the proof we refer to \cite[Lemma 2.10]{BMSW2}.
\end{proof}

We now recall the dimension-free estimates for the multipliers $m^{B^q_R}(\xi):=|B^q_R|^{-1}\mathcal F(\ind{B^q_R})(\xi)$ for $\xi\in\RR^d$.

\begin{lemma}{\cite[Lemma 2.11]{BMSW2}}
\label{lem:19}
There exist constants $c_q,C>0$ independent of $d$ and such that for
every $R>0$ and $\xi\in\RR^d$ we have
\begin{equation*}
|m^{B^q}(\xi)|\leq C (c_q R d^{-1/q} |\xi|)^{-1},
\quad \text{ and }\quad
|m^{B^q}(\xi)-1|\leq C (c_q Rd^{-1/q} |\xi|).
\end{equation*}
\end{lemma}

Lemma \ref{lem:11} and Lemma \ref{lem:19} are essential in proving the following estimate.

\begin{lemma}
	\label{lem:20}
	There exists a constant $C_q>0$ such that for every  $\delta\in(0, 1/2)$ and for all
	$r\in\NN$ and $R>0$ satisfying $1\le r\le R^{\delta}$ we have  
	\begin{align*}
	|\mm^{B^q, (r)}_R(\eta)|\le C_q \big(\kappa(r, R)^{-\frac{1}{3}+\frac{2\delta}{3}}
	+r\kappa(r, R)^{-\frac{1+\delta}{3}}+\big(\kappa(r, R)\|\eta\|\big)^{-1}\big)
	\end{align*}
	for every $\eta\in\TT^r$.
\end{lemma}
\begin{proof}
The inequality is obvious when $R\le 16,$ so it suffices to consider $R> 16.$  	

Firstly, we assume that
$\max\{\|\eta_1\|,\ldots,\|\eta_r\|\}>\kappa(r, R)^{-\frac{1+\delta}{3}}$.  Let $M=\big\lfloor
\kappa(r, R)^{\frac{2-\delta}{3}}\big\rfloor$ and assume without loss of generality that $\|\eta_1\|>\kappa(r, R)^{-\frac{1+\delta}{3}}$. Then
\begin{align}
\label{eq:48}
\begin{split}
|\mm^{B^q, (r)}_R(\eta)|&\le
\frac{1}{|B_{R}^{q,(r)}\cap\ZZ^r|}\sum_{x\in
	B_{R}^{q,(r)}\cap\ZZ^r}\frac{1}{M}\Big|\sum_{s=1}^Me^{2\pi i
	(x+se_1)\cdot\eta}\Big|\\
&\quad +\frac{1}{M}\sum_{s=1}^M\frac{1}{|B_{R}^{q,(r)}\cap\ZZ^r|}\Big|\sum_{x\in
	B_{R}^{q,(r)}\cap\ZZ^r}e^{2\pi i
	x\cdot\eta}-e^{2\pi i
	(x+se_1)\cdot\eta}\Big|.
\end{split}
\end{align}
Since $\kappa(r,R)\ge 1$ we now see that 
\begin{align}
\label{eq:49}
\frac{1}{M}\Big|\sum_{s=1}^Me^{2\pi i
	(x+se_1)\cdot\eta}\Big|\le M^{-1}\|\eta_1\|^{-1}\le
2\kappa(r, R)^{-\frac{1}{3}+\frac{2\delta}{3}}.  
\end{align}
We have assumed that $r\le R^{\delta}$, thus by Lemma \ref{lem:11}, with $z=se_1$
and $s\le M\le \kappa(r, R)^{\frac{2-\delta}{3}}$, we obtain
\begin{align}
\label{eq:50}
\begin{split}
\frac{1}{|B_{R}^{q,(r)}\cap\ZZ^r|}\Big| & \sum_{x\in B_{R}^{q,(r)}\cap\ZZ^r}e^{2\pi i x\cdot\eta} - e^{2\pi i
(x+se_1)\cdot\eta}\Big|\\
&\le\frac{1}{|B_{R}^{q,(r)}\cap\ZZ^r|}\big|\big(B_R^{q,(r)}\cap\ZZ^r\big)\triangle\big((se_1+B_R^{q,(r)})\cap\ZZ^r\big)\big|\\
&\le  8e\big(srR^{-1}e^{srR^{-1}}+e^{srR^{-1}}R^{-1+(q+1)\delta/q}\big)\\
&\le 16e^2\kappa(r, R)^{-\frac{1}{3}+\frac{2\delta}{3}},
\end{split}
\end{align}
since  $srR^{-1}\le \kappa(r, R)^{\frac{2-\delta}{3}}R^{-1+\delta}\le\kappa(r, R)^{-\frac{1}{3}+\frac{2\delta}{3}}\le1$ and $R^{-1+(q+1)\delta/q} \leq R^{-1+3\delta/2}\le R^{-\frac{1}{3}+\frac{2\delta}{3}}$, and  for $R>16$ we also have
$$
|B_{R}^{q,(r)}\cap\ZZ^r| \geq |B_{R-r^{1/q}}^{q,(r)}| \geq |B_{R-r^{1/2}}^{q,(r)}| = |B_R^{q,(r)}|\bigg(1-\frac{r^{1/2}}{R}\bigg)^r \geq |B_R^{q,(r)}|\big(1-r^{3/2}R^{-1}\big)\ge|B_R^{q,(r)}|/2.
$$ 
Combining \eqref{eq:48} with
\eqref{eq:49} and \eqref{eq:50} we obtain
\[
|\mm^{B^q, (r)}_R(\eta)|\le (16e^2+2)\kappa(r, R)^{-\frac{1}{3}+\frac{2\delta}{3}}.
\]

Secondly, we assume  that
$\max\{\|\eta_1\|,\ldots,\|\eta_r\|\}\le\kappa(r, R)^{-\frac{1+\delta}{3}}$.  Observe that
by \eqref{eq:40} we have
\begin{align*}
\bigg| \frac{1}{|B_{R}^{q,(r)}\cap\ZZ^r|}-
\frac{1}{|B_{R}^{q,(r)}|}\bigg|\le\frac{eR^{-1+(q+1)\delta/q}}{|B_{R}^{q,(r)}\cap\ZZ^r|}
\le\frac{2e\kappa(r, R)^{-\frac{1}{3}+\frac{2\delta}{3}}}{|B_{R}^{q,(r)}|}.
\end{align*} 
Then $|\mm^{B^q, (r)}_R(\eta)|$ is bounded by
\begin{align}
\label{eq:42}
\begin{split} 
\Big|\mm^{B^q, (r)}_R(\eta)&-\frac{1}{|B_{R}^{q,(r)}|}\mathcal
F(\ind{B_R^{q,(r)}})(\eta)\Big|+\frac{1}{|B_{R}^{q,(r)}|}\big|\mathcal
F(\ind{B_R^{q,(r)}})(\eta)\big|\\
&\le 2e\kappa(r, R)^{-\frac{1}{3}+\frac{2\delta}{3}}+\frac{1}{|B_{R}^{q,(r)}\cap\ZZ^r|}\Big|\sum_{x\in
	B_{R}^{q,(r)}\cap\ZZ^r}e^{2\pi i
	x\cdot\eta}-\int_{B_R^{q,(r)}}e^{2\pi i
	y\cdot\eta}\dif y\Big| \\ & \quad + \frac{1}{|B_{R}^{q,(r)}|}\big|\mathcal
F(\ind{B_R^{q,(r)}})(\eta)\big|.
\end{split}
\end{align}
Let $Q^{(r)}=[-1/2, 1/2]^r$ and note that by Lemma \ref{lem:11} with
$z=t\in Q^{(r)}$ we obtain
\begin{align}
\label{eq:51}
\begin{split}
\frac{1}{|B_{R}^{q,(r)}\cap\ZZ^r|}\Big| & \sum_{x\in
	B_{R}^{q,(r)}\cap\ZZ^r}e^{2\pi i
	x\cdot\eta}-\int_{B_R^{q,(r)}}e^{2\pi i
	y\cdot\eta}\dif y\Big|\\
&=\frac{1}{|B_{R}^{q,(r)}\cap\ZZ^r|}\Big|\sum_{x\in
	\ZZ^r}\int_{Q^{(r)}}e^{2\pi i
	x\cdot\eta}\ind{B_R^{q,(r)}}(x)-e^{2\pi i
	(x+t)\cdot\eta}\ind{B_R^{q,(r)}}(x+t)\dif t\Big|\\
&\le\frac{1}{|B_{R}^{q,(r)}\cap\ZZ^r|}\int_{Q^{(r)}}\big|(B_R^{q,(r)}\cap\ZZ^r)\triangle\big((t+B_R^{q,(r)})\cap\ZZ^r\big)\big|\dif
t\\
&\quad +
\frac{1}{|B_{R}^{q,(r)}\cap\ZZ^r|}\sum_{x\in
	\ZZ^r}\ind{B_R^{q,(r)}}(x)\int_{Q^{(r)}}|e^{2\pi i
	x\cdot\eta}-e^{2\pi i
	(x+t)\cdot\eta}|\dif t\\
& \le16e^2\kappa(r, R)^{-\frac{1}{3}+\frac{2\delta}{3}}+2\pi\big(\|\eta_1\|+\ldots+\|\eta_r\|\big)\\
&\le16e^2\kappa(r, R)^{-\frac{1}{3}+\frac{2\delta}{3}}+2\pi r\kappa(r, R)^{-\frac{1+\delta}{3}}.
\end{split}
\end{align}
Finally, by Lemma \ref{lem:19} we obtain
\begin{align*}
\frac{1}{|B_R^{q,(r)}|}|\mathcal F(\ind{B_R^{q,(r)}})(\eta)|\le  C\big(c_q\kappa(r, R)\|\eta\|\big)^{-1}.
\end{align*}
Combining this with \eqref{eq:42} and \eqref{eq:51} we conclude
\[
|\mm^{B^q, (r)}_R(\eta)|\le(16e^2+2e)\kappa(r, R)^{-\frac{1}{3}+\frac{2\delta}{3}}+2\pi r\kappa(r, R)^{-\frac{1+\delta}{3}}
+C c_q^{-1}\big(\kappa(r, R)\|\eta\|\big)^{-1},
\]
which completes the proof.
\end{proof}

\begin{lemma}
	\label{lem:13}
	For every $\delta\in(0, 1/2)$ and $\varepsilon\in(0, 1/(50q)]$ there is a
	constant $C_{q,\delta, \varepsilon}>0$ such
	that for every $d, N\in\NN$, if $r$ is an integer such that $1\le r\le d$ and 
	$\max\{1, \varepsilon^{\frac{(q+1)\delta}{q}}\kappa(d, N)^{\delta}/2\}\le r\le \max\{1,\varepsilon^{\frac{(q+1)\delta}{q}}\kappa(d, N)^{\delta}\}$,
	then for every $\xi=(\xi_1,\ldots,\xi_d)\in\TT^d$ we have
	\begin{align*}
	|\mm^{B^q}_N(\xi)|\le C_{q,\delta, \varepsilon}\big(\kappa(d, N)^{-\frac{1}{3}+\frac{2\delta}{3}}+(\kappa(d, N)\|\eta\|)^{-1}\big),
	\end{align*}
	where $\eta=(\xi_1,\ldots, \xi_r)$.
\end{lemma}
\begin{proof}
	If $\kappa(d, N)\le \varepsilon^{-\frac{q+1}{q}}$, then there is nothing to do, since the implied constant in question is allowed to depend on $q$, $\delta$, and $\varepsilon$. We will assume that $\kappa(d, N)\ge \varepsilon^{-\frac{q+1}{q}}$, which ensures that $\kappa(d, N)\ge10$.
	In view of Lemma \ref{lem:9} we have 
	\begin{align*}
	|\mm^{B^q}_N(\xi)|\le\sup_{R\ge \varepsilon^{(q+1)/q}\kappa(d, N) r^{1/q}}|\mm_{R}^{B^q,(r)}(\eta)|+4e^{-\frac{\varepsilon r}{10}},
	\end{align*}
	where $\eta=(\xi_1,\ldots,\xi_r)$. By Lemma \ref{lem:20}, since $r\le \varepsilon^{\frac{(q+1)\delta}{q}}\kappa(d, N)^{\delta}\le \kappa(r, R)^{\delta}\le R^{\delta}$, we obtain
	\begin{align*}
	|\mathfrak
	m_{R}^{B^q,(r)}(\eta)|\lesssim_q \kappa(r, R)^{-\frac{1}{3}+\frac{2\delta}{3}}
	+r\kappa(r, R)^{-\frac{1+\delta}{3}}+\big(\kappa(r, R)\|\eta\|\big)^{-1}.
	\end{align*}
	Combining the two estimates above with our assumptions we obtain the desired claim. 
\end{proof}

We have prepared all necessary tools to
prove inequality \eqref{eq:23}. We shall be working under the
assumptions of Lemma \ref{lem:13} with $\delta=2/7$.

\begin{proof}[Proof of Proposition \ref{prop:2}]
Assume that $\varepsilon=1/(50q)$. If
$\kappa(d, N)\le2^{\frac{7}{2}}\cdot (50q)^{\frac{q+1}{q}}$ then
clearly \eqref{eq:23} holds. Therefore, we can assume that
$d, N\in\NN$ satisfy
$2^{\frac{7}{2}}\cdot(50q)^{\frac{q+1}{q}} \le \kappa(d, N)\le 50qd^{1-1/q}$. We
choose an integer $1\le r\le d$ satisfying
$(50q)^{-\frac{2(q+1)}{7q}}\kappa(d, N)^{\frac{2}{7}}/2\le r\le (50q)^{-\frac{2(q+1)}{7q}}\kappa(d, N)^{\frac{2}{7}}$,
this is possible since
$(50q)^{-\frac{2(q+1)}{7q}}\kappa(d, N)^{\frac{2}{7}} \geq 2$ and
$(50q)^{-\frac{2(q+1)}{7q}}\kappa(d, N)^{\frac{2}{7}} \le d^{\frac{2(1-1/q)}{7}}\le d$.
By symmetry we may also assume that $\|\xi_1\|\ge\ldots\ge\|\xi_d\|$
and we shall distinguish two cases. Suppose first that
\begin{align*}
\|\xi_1\|^2+\ldots+\|\xi_r\|^2\ge\frac{1}{4}\|\xi\|^2.
\end{align*}
Then in view of Lemma \ref{lem:13} (with $\delta=2/7$ and
$r\simeq_q \kappa(d, N)^{\frac{2}{7}}$) we obtain
\begin{align*}
|\mathfrak m^{B^q}_N(\xi)|\le C_q\big(\kappa(d, N)^{-\frac{1}{7}}+(\kappa(d, N)\|\xi\|)^{-1}\big),
\end{align*}
and we are done. So we can assume that
\begin{align}
\label{eq:54}
\|\xi_1\|^2+\ldots+\|\xi_r\|^2\le\frac{1}{4}\|\xi\|^2.  
\end{align}
Let $\varepsilon_1=1/10$ and assume first that
\begin{align}
\label{eq:55}
\|\xi_j\|\le\frac{\varepsilon_1^{1/q}}{10\kappa(d,N)}\quad\text{ for all } \quad r\le
j\le d.
\end{align}
We use the symmetries of $B^q_N\cap\ZZ^d$ to write
\begin{align*}
\mm^{B^q}_N(\xi)=\frac{1}{|B^q_N\cap\ZZ^d|}\sum_{x\in B^q_N\cap\ZZ^d}\prod_{j=1}^d \cos(2\pi x_j \xi_j).
\end{align*}
Applying the Cauchy--Schwarz inequality, $\cos^2(2\pi x_j \xi_j)=1-\sin^2(2\pi x_j \xi_j)$ and $1-x\le e^{-x}$,  we obtain
\begin{align}
\label{eq:56}
\begin{split}
|\mm^{B^q}_N(\xi)|^2\le \frac{1}{|B^q_N\cap\ZZ^d|}\sum_{x\in B^q_N\cap\ZZ^d}\exp\Big(-\sum_{j=r+1}^d\sin^2(2\pi x_j \xi_j)\Big).
	\end{split}
	\end{align}
	For $x\in B^q_N\cap\ZZ^d$ we define
	\begin{align*}
	I_x&=\{i\in\NN_d : \varepsilon\kappa(d, N)\le |x_i|\le 2\varepsilon_1^{-1/q}\kappa(d,N) \},\\
	I_x'&=\{i\in\NN_d : 2\varepsilon_1^{-1/q}\kappa(d,N)< |x_i| \},\\
	I_x''&=\{i\in\NN_d : \varepsilon\kappa(d,N)\le |x_i|\}=I_x\cup I_x'.
	\end{align*}
	and
	\[
	E=\big\{x\in B^q_N\cap\ZZ^d : |I_x|\ge\varepsilon_1 d/2\big\}.
	\]
	Observe that
	\begin{align*}
	E^{\bf c}=&\big\{x\in B^q_N\cap\ZZ^d : |I_x|<\varepsilon_1 d/2\big\}
	=\big\{x\in B^q_N\cap\ZZ^d : |I_x''|<\varepsilon_1 d/2+|I_x'|\big\}\\
	\subseteq&\big\{x\in B^q_N\cap\ZZ^d : |I_x''|<\varepsilon_1 d/2+|I_x'|\text{
		and } |I_x'|\le \varepsilon_1 d/2\big\}
	\cup
	\big\{x\in B^q_N\cap\ZZ^d : |I_x'|> \varepsilon_1 d/2\big\}.
	\end{align*}
	Then it is not difficult to see that
	\[
	E^{\bf c}\subseteq \big\{x\in B^q_N\cap\ZZ^d : |I_x''|<\varepsilon_1 d\big\},
	\]
	since $  \big\{x\in B^q_N\cap\ZZ^d : |I_x'|> \varepsilon_1 d/2\big\}=\emptyset$.
	Then by Lemma \ref{lem:5} with $\varepsilon_2=\varepsilon$,
	we obtain
	\begin{align*}
	|E^{\bf c}|\le |\big\{x\in B^q_N\cap\ZZ^d : |I_x''|<\varepsilon_1 d\big\}|\le 2e^{-\frac{d}{10}}|B^q_N\cap\ZZ^d|.
	\end{align*}
	Therefore, by \eqref{eq:56} we have
	\begin{align*}
	|\mm^{B^q}_N(\xi)|^2
	&\le \frac{1}{|B^q_N\cap\ZZ^d|}\sum_{x\in
		B^q_N\cap\ZZ^d}\exp\Big(-\sum_{j\in I_x \cap J_r}\sin^2(2\pi x_j
	\xi_j)\Big)\ind{E}(x)+2e^{-\frac{d}{10}},
	\end{align*}
	where $J_r=\{r+1,\ldots, d\}$. Using \eqref{eq:103} and definition of $I_x$ we have
	\[
	\sin^2(2\pi x_j \xi_j)\ge 16|x_j|^2\|\xi_j\|^2\ge 16\varepsilon^2\kappa(d,N)^2\|\xi_j\|^2,
	\]
	since $2|x_j|\|\xi_j\|\le 1/2$ by \eqref{eq:55}, and consequently we obtain 
	\begin{align}
	\label{eq:58}
	\begin{split}
	\frac{1}{|B^q_N\cap\ZZ^d|}&\sum_{x\in
		B^q_N\cap\ZZ^d}\exp\Big(-\sum_{j\in I_x\cap J_r}\sin^2(2\pi x_j
	\xi_j)\Big)\ind{E}(x)\\
	&\le
	\frac{1}{|B^q_N\cap\ZZ^d|}\sum_{x\in
		B^q_N\cap\ZZ^d\cap E}\exp\Big(-16\varepsilon^2\kappa(d,N)^2\sum_{j\in I_x\cap
		J_r}\|\xi_j\|^2\Big)\le Ce^{-c\kappa(d,N)^2\|\xi\|^2}
	\end{split}
	\end{align}
	for some constants $C, c>0$. In order to get the last inequality in \eqref{eq:58} observe that 
	\begin{align*}
	\frac{1}{|B^q_N\cap\ZZ^d|}&\sum_{x\in B^q_N\cap\ZZ^d\cap E}
	\exp\Big(-16\varepsilon^2\kappa(d,N)^2\sum_{j\in I_x\cap J_r}\|\xi_j\|^2\Big)\\
	&=\frac{1}{|B^q_N\cap\ZZ^d|}\sum_{x\in
		B^q_N\cap\ZZ^d\cap E}\frac{1}{d!}\sum_{\sigma\in{\rm Sym}(d)}\exp\Big(-16\varepsilon^2\kappa(d,N)^2\sum_{j\in \sigma(I_x)\cap
		J_r}\|\xi_j\|^2\Big)\\
	&=\frac{1}{|B^q_N\cap\ZZ^d|}\sum_{x\in
		B^q_N\cap\ZZ^d\cap E}\mathbb E\bigg[\exp\Big(-16\varepsilon^2\kappa(d,N)^2\sum_{j\in \sigma(I_x)\cap
		J_r}\|\xi_j\|^2\Big)\bigg],
	\end{align*}
	since $\sigma\cdot (B^q_N\cap\ZZ^d\cap E)=B^q_N\cap\ZZ^d\cap E$ for every $\sigma\in{\rm Sym}(d)$. 
	We now apply Lemma \ref{lem:7} with $\delta_1=\varepsilon_1/2$, $d_0=r$, $I=I_x$, $\delta_0=3/5$, and
	\begin{displaymath}
	u_j = \left\{ \begin{array}{rl}
	16\varepsilon^2\kappa(d,N)^2 \|\xi_r\|^2 & \textrm{for } 1 \leq j \leq r, \\
	16\varepsilon^2\kappa(d,N)^2 \|\xi_j\|^2 & \textrm{for } r+1 \leq j \leq d, \end{array} \right.
	\end{displaymath} 
	noting that $16\varepsilon^2\kappa(d,N)^2 \|\xi_r\|^2 \leq 1/5$ by \eqref{eq:55}. We conclude that
	\begin{align*}
	\mathbb E\bigg[\exp\Big(-16\varepsilon^2\kappa(d,N)^2\sum_{j\in \sigma(I_x)\cap
		J_r}\|\xi_j\|^2\Big)\bigg]\le 3 \exp\Big(-c'\kappa(d,N)^2\sum_{j=r+1}^d\|\xi_j\|^2\Big),
	\end{align*}
	holds for some $c'>0$ and for all $x\in B^q_N\cap\ZZ^d\cap E$.
	This proves \eqref{eq:58} since by
	\eqref{eq:54} we obtain
	\begin{align*}
	\exp\Big(-c'\kappa(d,N)^2\sum_{j=r+1}^d\|\xi_j\|^2\Big)\le \exp\Big(-\frac{c'\kappa(d,N)^2}{4}\sum_{j=1}^d\|\xi_j\|^2\Big).
	\end{align*}
	
	Assume now that \eqref{eq:55} does not hold. Then
	\begin{align*}
	\|\xi_j\|\ge\frac{\varepsilon_1^{1/2}}{10\kappa(d,N)}\quad\text{ for all
	} \quad 1\le j\le r.
	\end{align*}
	Hence
	\begin{align*}
	\|\xi_1\|^2+\ldots+\|\xi_r\|^2\ge\frac{\varepsilon_1r}{100\kappa(d,N)^2}.
	\end{align*}
	Therefore, we invoke Lemma \ref{lem:13} with $\eta=(\xi_1,\ldots, \xi_r)$ again and obtain 
	\begin{align*}
	|\mm^{B^q}_N(\xi)|&\lesssim_q \kappa(d, N)^{-\frac{1}{7}}+(\kappa(d, N)\|\eta\|)^{-1}\\
	&\lesssim_q \kappa(d, N)^{-\frac{1}{7}},
	\end{align*}
	since $r\simeq_q \kappa(d, N)^{\frac{2}{7}}$. This completes the proof of Proposition \ref{prop:2}.
\end{proof}

\end{document}